\documentclass{gtpart}
\usepackage{amsmath, amsbsy, amsthm, amscd, graphicx}
\usepackage{amssymb, latexsym, color}
\usepackage{pinlabel}
\usepackage[all]{xy}

\title{Fixed-point property of random quotients by plain words}

\author{Hiroyasu Izeki}
\givenname{Hiroyasu}
\surname{Izeki}
\address{Department of Mathematics, Keio University, 
3-14-1 Hiyoshi, Kohoku-ku, Yokohama, 223-8522, Japan}
\email{izeki@math.keio.ac.jp}

\thanks{Suupported by the Grant-in-Aid
for Scientific Research, The Ministry of Education,
Science, Sports and Culture, Japan, No.~21540062.}



\numberwithin{equation}{section}

\theoremstyle{plain}
\newtheorem{Theorem}{Theorem}[section]
\newtheorem{Proposition}[Theorem]{Proposition}
\newtheorem{Lemma}[Theorem]{Lemma}

\theoremstyle{definition}
\newtheorem{Definition}[Theorem]{Definition}

\theoremstyle{remark}
\newtheorem{Remark}{{\bf Remark}}

\newcommand{\N}{{\mathbb N}}
\newcommand{\isom}{\mathop{\rm Isom}\nolimits}
\newcommand{\cat}{\mathrm{CAT}(0)}

\begin{document}

\begin{abstract}
We show a fixed-point property of certain random groups for a
wide class of $\cat$ spaces. The model of random groups under
consideration is given as the set of presentations $(S,R)$, where $S$ is
a generating set and the set of relations $R$ is a subset of the set of
all plain words of the same length with suitably fixed density.  
Our main theorem says that groups obtained by such presentations have
fixed-point property for all $\cat$ spaces having bounded singularities
with high probability. 
\end{abstract}

\maketitle

\section{Introduction}

Let $Y$ be a metric space and $\isom (Y)$ the group of isometries of
$Y$. We say a group $\Gamma$ has fixed-point property for $Y$ if any
isometric action of $\Gamma$ on $Y$ admits a fixed point; namely, 
for any homomorphism $\rho \colon \Gamma \longrightarrow \isom (Y)$,
there is a point $p \in Y$ such that $\rho(\gamma)p=p$ for all 
$\gamma \in \Gamma$.  
When we take a Hilbert space $\mathcal{H}$ as $Y$,
according to theorems due to Delorme and Guichardet, the
fixed-point property for $\mathcal{H}$ is equivalent to the Kazhdan's
Property (T).  Groups with Kazhdan's Property (T) have interesting
and misterious feature as described in \cite{Bekka-delaHarpe-Valette}. 
On the other hand, a group with the fixed-point property for trees
cannot be decomposed into an amalgamated free product as
explained in \cite{Serre}. 

Looking at these examples, one may suspect that having such a
fixed-point property imposes a strong restriction on groups. 
Recent progress suggests that, however,  groups with fixed-point 
property are distributed densely in a certain class of finitely
generated groups as shown in \cite{Gromov3}, \cite{Silberman},
\cite{Zuk2},  \cite{Izeki-Kondo-Nayatani1}, 
\cite{Izeki-Kondo-Nayatani2}, and \cite{Fujiwara-Toyoda}. 
In these papers, we see that certain random groups have strong
fixed-point property: random groups in the triangle model (\cite{Zuk2}, 
\cite{Izeki-Kondo-Nayatani1}) and the graph model (\cite{Gromov3},
\cite{Silberman}, \cite{Izeki-Kondo-Nayatani2}, and
\cite{Fujiwara-Toyoda}).  Here, models of 
random groups are provided as certain sets of group presentations, and
the phrase ``random groups in the model admit fixed-point property''
means that groups given by presentations in the model admit
fixed-point property with high probability. 
In a slightly different direction, there is a result on so-called marked
groups presented in \cite{Kondo}, which claims that groups with
fixed-point property is dense in the closure of the set of marked groups
which are torsion-free, infinite, and non-elementary hyperbolic
with respect to a certain 
topology.  Also we should mention that the 
existence of a group with very strong fixed-point property is shown in 
\cite{Fisher-Silberman}  and \cite{Naor-Silberman}. 
Along the same lines as these results, our main theorems in the
present paper say that groups with fixed-point property for certain
metric spaces in fact form an extensive class among 
finitely presented groups; random quotients by plain words with certain
densities admit strong fixed-point property as we explain below. 

Let $m$, $\eta$ be natural numbers and $0 < d < 1$ a real number. 
Denote by $S$ a set consisting of $m$ letters and their inverses:
$S=\{s_1, s_1^{-1}, \dots, s_m, s_m^{-1} \}$.  We call a sequence 
of letters belonging to $S$ a {\it word}, and the number of letters
forming a word the {\it length} of the word. 
Thus, any word with length $\eta$ takes the form of 
$s_{i_1}^{\epsilon_1}s_{i_2}^{\epsilon_2}\dots 
s_{i_{\eta}}^{\epsilon_{\eta}}$,
where $\epsilon_j=\pm 1$ and $s_{i_j} \in S$. 
We denote by $W_{\eta}$ the set of words (plain words) of length $\eta$. 
We consider a group generated by $S$ with a presentation $P=(S,R)$, where
the set of relations $R$ is a subset of $W_{\eta}$. 
We note that a relation $r \in R$ may be
a reducible one,  namely, $r$ may contain a sequence of the form 
$s_i s_i^{-1}$ or $s_j^{-1}s_j$. 
Fix a real number $c>1$ and set
\begin{equation*}
 P(m,\eta,d)=\{P=(S,R) \mid R \subset W_{\eta} \text{ and }
        c^{-1}(2m)^{d\eta} \leq \# R \leq c(2m)^{d\eta}\}. 
\end{equation*}
Let $\Gamma$ be the free group generated by $S$: 
$\Gamma=\langle s_1 \rangle \ast \dots \ast \langle s_m \rangle$. 
Then each element in $W_{\eta}$ determines an element in $\Gamma$ in an
obvious way. 
Denote by $\Gamma_P$ the group given by a presentation $P=(S,R)$:
$\Gamma_P = \Gamma / \overline{R}$, where $\overline{R}$ is the normal
closure of $R$ in $\Gamma$.  It is an easy exercise to show that, for
any finitely presented group $\Gamma$, there exists
$m, \eta\in \N$ and $0<d<1$ such that $\Gamma =\Gamma_P$ for some 
$P \in P(m,\eta,d)$; that is, our model contains all finitely presented
groups. 

Let $Y$ be a $\cat$ space, a metric space with nonpositive 
curvature (see \S \ref{sec:energy-equiv-maps} for the definition).
Complete, simply connected Riemannian manifolds with 
nonpositive sectional curvature, trees, and Hibert spaces are typical
and important examples of $\cat$ spaces. 
When $\Gamma$ has fixed-point property for $Y$, we say $\Gamma$ has
$F(Y)$ in what follows. Note that if $\Gamma$ is a finite group, then
$\Gamma$ has fixed-point property for any $\cat$ space. In fact, an
orbit of any isometric action of $\Gamma$ on $Y$ consists of finite
points, and hence, in $\cat$ space $Y$, we can find the barycenter of
the orbit that must be invariant under the action of $\Gamma$. 
(See Proposition \ref{barycenter} for the definition and the existence
of the barycenter.) 
Thus, in what follows, our
interest is in an infinite group $\Gamma$ with fixed-point property for
$\cat$ spaces. 

If we take $d$ in $P(m,\eta,d)$ larger, then 
producing homomorphism from $\Gamma_P$ for $P \in P(m,\eta,d)$ into
$\isom (Y)$ becomes more difficult, which means that isometric
actions of $\Gamma_P$ on $Y$ must carry strong restriction, in general. 
This leads us to a naive consideration that the larger $d$ becomes, 
the more we can expect  $\Gamma_P$ for $P \in P(m,\eta,d)$ to have
fixed-point property.  On the other hand, if we take $d$ too large, then
$R$ becomes large; it may force 
$\Gamma_P$ to be a finite group, which is the case we want to avoid.  
Our main theorems tell us that the truth is exactly as this
consideration describes, and there does exist $d$ which provides
infinite groups with fixed-point property for a certain class of $\cat$
spaces with high probability. 
For Hilbert spaces, we will show the following theorem. 

\begin{Theorem}
 \label{thm:main_Hilbert}
 Let $\mathcal{H}$ be a Hilbert space.  For
 $1/3<d<1-\frac{1}{2}\log_{2m}(8m-4)$, 
\begin{equation*}
 \lim_{\eta\to \infty} 
 \frac{\#\{P \in P(m,\eta,d) \mid 
        \Gamma_P \text{ is an infinite group having }F(\mathcal{H})\}}
      {\# P(m,\eta,d)}= 1
\end{equation*}
 holds. 
\end{Theorem}

If we only allow irreducible words to be relations in the
definition of $P(m, \eta ,d)$, we obtain the density model
of Gromov.  One can expect a result for the density model similar
to Theorem~\ref{thm:main_Hilbert}, and an idea of the proof, a reduction
to Zuk's theorem (\cite[Theorem 4]{Zuk2}), is suggested in
\cite[I.3.g]{Ollivier0}\footnote{After the present paper was submitted, 
a complete proof was given by 
Marcin Kotowski and Micha{\l} Kotowski \cite{Kotowsiki}.}. 

In order to state our result for general $\cat$ spaces, we need to
introduce an invariant of $\cat$ spaces defined in
\cite[Definition 6.1]{Izeki-Nayatani}: Let $\mu$ be a
measure on a $\cat$ space $Y$, whose support $\mathrm{supp}\, \mu$
consists of finite points.  
Denote by $\overline{\mu}$ the barycenter of $\mu$ (see
Proposition~\ref{barycenter}), and let
\begin{equation*}
 \delta(\mu) = \inf \left\{ \left.
  \frac{\left\vert \displaystyle{\int_Y \varphi(p) \ d\mu(p)} \right\vert^2}
   {\displaystyle{\int_Y \vert \varphi(p) \vert^2 \ d\mu(p)} } \right|
 \begin{matrix}
  \varphi \colon \mathrm{supp}\, \mu \longrightarrow \mathcal{H}, \quad 
   \varphi(\overline{\mu})= \mathbf{0}, \hskip2.5cm \\ 
   \vert \varphi(p) \vert = d(\overline{\mu}, p) \text{ for any } p \in
  \mathrm{supp}\, \mu, \hskip1.5cm \\
   \vert \varphi(p) - \varphi(q) \vert \leq d(p,q) \text{ for any }
   p,q \in \mathrm{supp}\, \mu
 \end{matrix}  
  \right\}, 
\end{equation*}
where $\mathcal{H}$ is a Hilbert space, and $\mathbf{0}$ denotes its
origin. Then $\delta(Y)$ is defined to be
\begin{equation*}
 \delta(Y) = \sup \left\{
  \delta(\mu) \mid \mu \text{ is a measure on }Y 
   \text{ with }\# \mathrm{supp}\, \mu < \infty
  \right\}. 
\end{equation*}
This invariant $\delta(Y)$
measures the singularity of $Y$, and obviously it takes values in
$[0,1]$. 
If $\delta(Y)=0$, then $Y$ is considered to be almost regular;
$\cat$ Riemannian manifolds and trees satisfy $\delta(Y)=0$. 
For general $\cat$ spaces we will show: 

\begin{Theorem}
\label{thm:main}
 Take $0\leq \delta<1$ and an integer $k$ so that $k>1/(1-\delta)$.  
 Let $\mathcal{Y}_{\delta}$ be the class of $\cat$ spaces with 
 $\delta(Y) \leq \delta$. 
 Then, for $k/(2k+1)<d<1-\frac{1}{2}\log_{2m}(8m-4)$, 
\begin{equation*}
 \lim_{\eta\to \infty} 
 \frac{\#\{P \in P(m,\eta,d) \mid 
        \Gamma_P \text{ is an infinite group having }F(Y) 
        \text{ for all }Y  \in \mathcal{Y}_{\delta}\}}
      {\# P(m,\eta,d)}= 1
\end{equation*}
 holds. 
\end{Theorem}

 A part of the theorems above is due to Ollivier; 
 Ollivier \cite{Ollivier} proved that, by choosing 
 $d< 1-\frac{1}{2} \log_{2m}(8m-4)$, 
\begin{equation*}
 \lim_{\eta\to \infty} 
 \frac{\#\{P \in P(m,\eta,d) \mid \Gamma_P 
           \text{ is non-elementary hyperbolic}\}}
      {\# P(m,\eta,d)} = 1, 
\end{equation*}
 where $\Gamma$ is non-elementary hyperbolic means that $\Gamma$ is
 infinite hyperbolic group (in the sense of Gromov) which does not
 contain $\Z$ as a finite index subgroup. 
 Since $\frac{1}{2}\log_{2m}(8m-4)\to 1/2$ as $m\to \infty$, 
 by taking $m$ sufficiently large, 
 Theorem \ref{thm:main} tells us that, for any $0 \leq \delta <1$, 
 we can choose $d$ so that 
 infinite hyperbolic groups with $F(\mathcal{Y}_{\delta})$ form an 
 extensive class among the groups coming from $P(m,\eta,d)$. Here a
 group $\Gamma$ has $F(\mathcal{Y}_{\delta})$ means that $\Gamma$ has
 $F(Y)$ for all $Y \in \mathcal{Y}_{\delta}$.  We note that, in
 \cite{Fujiwara-Toyoda}, it is shown that any $\cat$ cube complex $Y$
 satisfies $\delta(Y) \leq 1/2$.  Therefore $\mathcal{Y}_{1/2}$ contains
 all $\cat$ Riemannian manifolds and $\cat$ cube complexes. On the other
 hand, there does exist $\cat$ space $Y$ with $\delta(Y)=1$
 (\cite{Kondo2}). 

 This paper is organized as follows. In \S \ref{sec:graph-assoc-pres},
 we introduce a finite graph associated to each element of $P(m,\eta,d)$
 and state Proposition \ref{prop:random_eigenvalue}, which concerns with
 the eigenvalue of random graphs associated to $P(m,\eta,d)$.  
 In \S \ref{sec:energy-equiv-maps}, we
 introduce an energy of equivariant maps from $\Gamma$ to $Y$, and give
 a criterion for an action of $\Gamma$ on $Y$ to have a fixed point in
 terms of the ratio of certain energies, which is a modification
 of \cite[Lemma 2.7(2)]{Izeki-Kondo-Nayatani2}.  
 We also state Proposition \ref{prop:random_Rayleigh_quotient}, which
 connects the ratio of energies and the eigenvalue of the graph
 introduced in \S \ref{sec:graph-assoc-pres}.  In \S \ref{sec:proof}, we
 give the proofs of Theorem \ref{thm:main_Hilbert} and \ref{thm:main},
 assuming Propositions \ref{prop:random_eigenvalue} and
 \ref{prop:random_Rayleigh_quotient}. After a brief review of some facts
 in probability theory in \S \ref{sec:some-facts-from}, we give the
 proofs of Propositions \ref{prop:random_eigenvalue} and
 \ref{prop:random_Rayleigh_quotient} in \S \ref{sec:conc-proofs-prop}.  
 As a matter of fact, our proofs give results on random quotients of any 
 group, which will be explained in \S \ref{sec:random-quotient}. 

 {\bf Acknowledgement.}  The author would like to thank Shin Nayatani
 and Takefumi Kondo for their interest in this work. The author is also
 grateful to Yann Ollivier for his valuable
 comments on the first manuscript of this paper.

\section{A graph associated to a presentation}
\label{sec:graph-assoc-pres}

Fix a natural number $k$.  For each natural number $\eta \geq k$, we
take $l$ so that $l$ is the largest integer
satisfying $\eta \equiv l$ (mod $2$) and $l \leq \eta/(2k+1)$.  
We set $2n = \eta -l$.  Therefore we have
\begin{equation}
 \label{eq:k_and_l}
 \begin{split}
    (2k+1)l \leq & \eta=(2n+l) \leq (2k+1)l + 2(2k+1), \\  
    kl \leq & n \leq kl + (2k+1). 
 \end{split}
\end{equation}
Note that 
$\eta \to \infty$ implies both $l \to \infty$ and $n\to \infty$. 

Let $0<d<1$ and $P \in P(m,\eta,d)$. 
We associate a finite graph $G_P$ to
$P=(S,R)$ as follows: the vertex set $V$ of $G_P$ is identified with $W_n$ 
and a pair $\{v_1,v_2\} \in V \times V$ is 
joined by an edge if there is a word 
$v_1uv_2^{-1} \in R \subset W_{\eta}$, where $u \in W_l$ and $v^{-1}$
denotes the inverse word 
$s_{i_n}^{-\epsilon_n} \dots s_{i_1}^{-\epsilon_1}$ of 
$v=s_{i_1}^{\epsilon_1} \dots s_{i_n}^{\epsilon_n}$, $\epsilon_j=\pm 1$. 
We denote the set of (unoriented) edges of $G_P$ by $E_P$.  Thus
$G_P=(V,E_P)$, where 
\begin{equation*}
 \begin{cases}
  V = W_{n}, \\
  E_P = \{\{v_1,v_2\} \mid v_1 u v_2^{-1} \in R 
           \text{ for some } u \in W_{l}\}.  
 \end{cases}
\end{equation*}
If there are two words $u,u' \in W_{l}$ such
that $r=v_1uv_2^{-1}, r'=v_1u'v_2^{-1} \in R$, 
then $r$ and $r'$ define distinct edges with the same endpoints 
$v_1, v_2 \in V$;  thus $G_P$ may have multiple edges.  Also $G_P$ may
have self loops, since $r=vuv^{-1}$ may appear in $R$.
Let $\nu$ be a measure on $V \times V$ defined by
\begin{equation*}
 \nu(v_1,v_2) = \frac{\#\{u \in W_l \mid v_1uv_2^{-1} \in R\} 
  + \#\{u' \in W_l \mid v_2 u' v_1^{-1} \in R\}}{2\# R}. 
\end{equation*}
By definition, $\nu$ is a symmetric probability measure on 
$V \times V$ and can be regarded as a probability measure
on $E_P$. By a slight abuse of notation, we denote by $\nu$ a
probability measure (or simply a funtion) on $V$  defined as
\begin{equation*}
  \nu(v)= \sum_{v'} \nu(v,v') 
        = \frac{\#\{u \in W_{n+l} \mid vu \in R\} 
  + \#\{u' \in W_{n+l} \mid u' v^{-1} \in R\}}{2\# R}. 
\end{equation*}
Note that
\begin{equation*}
 \nu(v)= \frac{\mathrm{deg}(v)}{2\# R} 
       = \frac{\mathrm{deg}(v)}{2\# E_P}, 
\end{equation*}
where $\mathrm{deg}(v)$ is the {\it degree} of $v$, which is defined to
be the cardinality (counted with multiplicity) of edges having $v$ as
one of the endpoints. Since a self loop comes from a relation of the
form $vuv^{-1}$, it is counted twice in the definition of our degree. 
For real valued functions $f$, $g$ defined on $V$, their inner product
is defined by 
\begin{equation*}
 (f,g)_{L^2}= \sum_{v \in V} \nu(v) f(v)g(v). 
\end{equation*}
The {\it Laplacian} $\Delta_P$ of $G_P$ is defined by
\begin{equation*}
 \Delta_P f(v)  = f(v) - \sum_{v'} \frac{\nu(v,v')}{\nu(v)} f(v'), 
\end{equation*}
for a funtion $f \colon V \longrightarrow \R$ on $V$. 
It is easy to see that $\Delta_P$ is a symmetric and nonnegative
operator with respect to the inner product $(\cdot, \cdot)_{L^2}$
defined above.  The eigenvalues of $\Delta_P$ are distributed in the
interval $[0,2]$, and it is clear that  
constant functions are eigenfunctions corresponding to the smallest
eigenvalue $0$. We denote by $\lambda_1(G_P)$ the second smallest
eigenvalue of $\Delta_P$ counted with multiplicity, which is positive if
and only if $G_P$ is connected. 
The {\it Rayleigh quotient} 
$RQ(f)$ of $f \colon V \longrightarrow \R$ is defined as 
\begin{equation*}
 RQ(f) = \frac{\Vert df \Vert^2_{L^2}}{F(f)}, 
\end{equation*}
where
\begin{equation*}
 \begin{cases}
   \Vert df \Vert^2_{L^2} = \frac{1}{2}\sum_{v \in V} \sum_{v' \in V}
    \nu(v,v') \vert f(v)-f(v'))\vert^2 \\
    F(f) = \frac{1}{2}\sum_{v,v' \in V} 
      \nu(v) \nu(v') \vert f(v)-f(v')\vert ^2.
 \end{cases}
\end{equation*}
Easy computations show that $F(f)$ is equal to the $L^2$-norm
of the component of $f$ which is perpendicular to constant functions 
and that $\Vert df \Vert^2_{L^2} = \left(f, \Delta_P f\right)_{L^2}$. 
Therefore we obtain 
\begin{equation*}
  \lambda_1(G_P)= \inf \{RQ(f) \mid f \text{ is a nonconstant
   function on }V\}. 
\end{equation*}
In \S \ref{proof_of_random_eigenvalue}, We will show the following
proposition on the eigenvalue of the Laplacian of $G_P$. 

\begin{Proposition}
 \label{prop:random_eigenvalue}
 Fix a natural number $k$. 
 If $k/(2k+1)<d$, then, for any $\varepsilon>0$, 
 \begin{equation*}
  \lim_{\eta\to \infty}
  \frac{\# \{P \in P(m,\eta,d) \mid 
                 \lambda_1(G_P) > 1-\varepsilon\}}
       {\# P(m,\eta,d)} = 1
 \end{equation*}
 holds. 
\end{Proposition}

\section{Energy of equivariant maps and fixed points}
\label{sec:energy-equiv-maps}

In this section, we review and improve some results obtained in
\cite{Izeki-Kondo-Nayatani1} and
\cite{Izeki-Kondo-Nayatani2}, which concern the energy of
equivariant maps from a finitely generated group to a $\cat$ space. 
A complete metric space $Y$ is called $\cat$ space if any geodesic
triangle is thinner than that in the Euclidean plane in the following sense. 

Let $Y=(Y,d_Y)$ be a complete metric space and $p,q \in Y$.
A {\it geodesic} joining $p$ to $q$ is 
an isometric embedding $c$ of a closed interval $[0,l]$ into
$Y$ such that $c(0)=p$ and $c(l)=q$, where $l=d_Y(p,q)$. 
If any pair of points in $Y$ are joined by a geodesic, then $Y$ is
called a {\it geodesic space}. 
Consider a triangle in a geodesic space $Y$ with vertices 
$p_1,p_2,p_3 \in Y$ and  
three geodesic segments $p_1p_2, p_2p_3, p_3p_1$ joining them. 
We denote this triangle by $\Delta(p_1,p_2,p_3)$ and
call such a triangle a {\it geodesic triangle}.  
Note that we can take 
a triangle $\Delta(\overline{p_1},\overline{p_2},\overline{p_3})$ in
$\R^2$ so that $d_Y(p_i, p_j)=d_{\R^2}(\overline{p_i},\overline{p_j})$,
since the side lengths of $\Delta(p_1,p_2,p_3)$ satisfy the triangle
inequality. 
We call $\Delta(\overline{p_1},\overline{p_2},\overline{p_3})$ a {\it
comparison triangle} for $\Delta(p_1,p_2,p_3)$.  And we call 
a point $\overline{q}\in \overline{p_i}\overline{p_j}$
a {\it comparison point} for $q\in p_ip_j$ if 
$d_Y(p_i,q) = d_{\R^2}(\overline{p_i}, \overline{q})$.
Then we say that $\Delta(p_1,p_2,p_3)$ satisfies the 
{\it $\cat$ inequality}, 
if $d_Y(q_1,q_2) \leq d_{\R^2}(\overline{q_1}, \overline{q_2})$ for any
pair of points $q_1, q_2$ on the sides of $\Delta(p_1,p_2,p_3)$ and
their comparison points $\overline{q_1}, \overline{q_2}$. 
If every geodesic triangle in $Y$ satisfies the $\cat$ inequality, 
then $Y$ is called a {\it $\cat$ space}. 

If $Y$ is a $\cat$ space, it
is easy to verify that $Y$ has two important properties: the uniqueness
of a geodesic connecting given two points and the contractibility. 
A complete, simply
connected Riemannian manifold of nonpositive sectional curvature is a
typical example of $\cat$ space.  Also trees and Hilbert spaces are
$\cat$ spaces. 
For an detailed exposition on $\cat$ space, we refer the reader to
\cite{Bridson-Haefliger}. 

Now we recall some definitions which will be necessary later.  

\begin{Definition}
\label{tangent cone}
Let $Y$ be a $\cat$ space.

(1)  Let $c$ and $c'$ be two nontrivial geodesics in $Y$ starting 
from $p \in Y$. The {\it angle} $\angle_p(c,c')$ between $c$ and $c'$ 
is defined by 
\begin{equation*}
 \angle_p(c,c')= \lim_{t,t' \rightarrow 0}
 \angle_{\overline{p}}(\overline{c(t)},\overline{c'(t')}),
\end{equation*}
 where
 $\angle_{\overline{p}}(\overline{c(t)},\overline{c'(t')})$ 
 denotes the angle between the sides
 $\overline{p}\overline{c(t)}$ and $\overline{p}\overline{c'(t')}$
 of the comparison triangle
 $\Delta(\overline{p},\overline{c(t)},\overline{c'(t')})
 \subset \R^2$.

(2)  Let $p \in Y$.
 For a pair of nontrivial geodesics $c$, $c'$ starting from $p$, we 
 define a relation $\sim$ by $c\sim c'$ if and only if 
 $\angle_p(c,c')=0$. 
 Then $\sim$ becomes an equivalence relation on the set of
 all nontrivial geodesics starting from $p$ denoted by 
 $\widetilde{(S_pY)^{\circ}}$. 
 Then the angle $\angle_p$ induces a distance on the quotient 
 $(S_pY)^{\circ} = \widetilde{(S_pY)^{\circ}}/\sim$, 
 which we denote by the same symbol $\angle_p$.  The completion 
 $(S_pY,\angle_p)$ of the metric space $((S_pY)^{\circ}, \angle_p)$ is
 called the {\it space of directions}  at $p$.

(3) Let $TC_pY$ be the cone over $S_pY$, namely,
\begin{equation*}
 TC_pY = (S_pY \times \R_{\geq 0}) / (S_pY \times \{0\}). 
\end{equation*}
 Let $W,W' \in TC_pY$. We may write $W=(V,t)$ and $W'=(V',t')$, where
 $V,V' \in S_pY$ and $t,t' \in \R_{\geq 0}$.   Then
\begin{equation*}
 d_{TC_pY}(W, W')= t^2 + {t'}^2 - 2tt'\cos \angle_p(V,V')
\end{equation*}
 defines a distance on $TC_pY$. The metric space $(TC_pY, d_{TC_pY})$ is
 known to be a $\cat$ space and is called the {\it tangent cone} at $p$.
 We define an ``inner product''  on  $TC_pY$ by
\begin{equation*}
 \langle W, W' \rangle = tt'\cos \angle_p(V,V').
\end{equation*}
 We often denote the length $t$ of $W$ by $|W|$; thus we have
 $|W|=\sqrt{\langle W,W \rangle}=d_{TC_pY}(0_p,W)$, where $0_p$ denotes
 the origin of $TC_pY$.   

(4) Define a map $\pi_p \colon Y \longrightarrow TC_pY$ by
 $\pi_p(q)=([c], d_Y(p,q))$, where $c$ is the geodesic joining $p$ to
 $q$ and $[c]\in S_pY$ is the equivalence class of $c$. Then $\pi_p$ is
 distance non-increasing.  
\end{Definition}

We often use the following important fact. 

\begin{Proposition}[{\cite[Lemma 2.5.1]{Korevaar-Schoen}}]
\label{barycenter}
 Let $Y$ be a $\cat$ space.  Suppose that $\mu$ is a finitely supported
 measure on $Y${\rm ;} that is, $\mu$ is a measure of the form 
 $\mu=\sum_{i=1}^m \mu_i \mathrm{Dirac}_{p_i}$, where $\mathrm{Dirac}_p$
 is the Dirac measure supported on a point $p \in Y$, 
 and $p_1, \dots, p_m \in Y$. 
 Then there exists a unique point $p_0 \in Y$
 which minimizes the function
\begin{equation*} 
 F(q) = \int_Y d(p,q)^2 \, d\mu(p) = \sum_{i=1}^m \mu_i d(p_i,q)^2, 
 \quad q \in Y.
\end{equation*}
\end{Proposition}
We call the point $p_0$ the {\it barycenter of} 
$\{p_1, \dots ,p_m\}$ {\it with respect to a measure} $\mu$, or
simply {\it barycenter of} $\mu$. 

\medskip

Let $Y$ be a $\cat$ space, and denote by $\isom(Y)$ the group of
isometries of $Y$. 
Let $\Gamma$ be a finitely generated group, and
$\rho\colon \Gamma\longrightarrow \isom(Y)$ a homomorphism.  
We call a map $f \colon \Gamma \longrightarrow Y$ 
$\rho$-{\it equivariant} if $f$ satisfies 
$f(\gamma \gamma')=\rho(\gamma)f(\gamma')$ for all 
$\gamma, \gamma'\in \Gamma$. 
Suppose that a $\Gamma$-invariant, symmetric random walk with
finite-support property on $\Gamma$ is given.  In other word, we are
given $\mu\colon \Gamma \times \Gamma \longrightarrow [0,1]$ with the
following properties: 
\begin{itemize}
 \item (finite-support property) 
 For any $\gamma \in \Gamma$, $\mu(\gamma, \gamma')= 0$ for
 all but finitely many $\gamma' \in \Gamma$, 
 \item  For any $\gamma \in \Gamma$, 
       $\displaystyle{\sum_{\gamma' \in \Gamma} \mu(\gamma, \gamma') = 1}$, 

 \item ($\Gamma$-invariance) 
 $\mu(\gamma \gamma', \gamma \gamma'')=\mu(\gamma',\gamma'')$
 for any $\gamma$, $\gamma'$, and $\gamma'' \in \Gamma$,
 \item (symmetry) $\mu(\gamma, \gamma')=\mu(\gamma',\gamma)$ for any 
        $\gamma, \gamma' \in \Gamma$. 
\end{itemize}
Then we define the {\it energy} $E_{\mu}(f)$ of a
$\rho$-equivariant map $f$ by 
\begin{equation}
 E_{\mu}(f)= \frac{1}{2} \sum_{\gamma' \in \Gamma} 
 \mu(\gamma,\gamma')d_Y(f(\gamma),f(\gamma'))^2. 
\end{equation}
Note that, due to the $\Gamma$-invariance of $\mu(\cdot,\cdot)$ and 
the $\rho$-equivariance of $f$, 
this value does not depend on the choice of $\gamma$. 
A $\rho$-equivariant map $f$ is said to be {\it harmonic} if $f$
minimizes $E_{\mu}$ among all $\rho$-equivariant maps. 
Note that the image of a $\rho$-equivariant map 
$f \colon \Gamma \longrightarrow Y$ is the $\rho(\Gamma)$-orbit of the
point $f(e)$, and $f$ is determined by the choice of $f(e)$. Therefore,
the set of all $\rho$-equivariant maps from $\Gamma$ to $Y$ can be
identified with $Y$.  Then the energy functional $E_{\mu}$ becomes a
convex continuous function on $Y$, since $Y$ is a $\cat$ space. 
Let $-\mathrm{grad}E_{\mu}(f)$ be the negative gradient of the
energy functional $E_{\mu}$ at $f$. When $Y$ is a Riemannian manifold
this should be understood as the negative of the ordinary gradient. In
general, one can give a reasonable definition of
$-\mathrm{grad}E_{\mu}(f)$ as an element of the tangent cone of $Y$, 
identified with the space of $\rho$-equivariant maps, at a point
$f$. In fact, $-\mathrm{grad}E_{\mu}(f)$ should be 
$2(-\Delta f(e))$, where $-\Delta f(e)$ is the barycenter of 
$\{F_{f(e)}(\gamma)=\pi_{f(e)}\circ f(\gamma) \mid \mu(e,\gamma)\not=
0\}$ 
with respect to a measure 
$\sum_{\gamma}\mu(e,\gamma)\mathrm{Dirac}_{F_{f(e)}(\gamma)}$. 
The following proposition gives a sufficient 
condition for the existence of a fixed point of $\rho(\Gamma)$ in terms
of the energy functional. 

\begin{Proposition}[{\cite[\S 1]{Izeki-Kondo-Nayatani1}}, see
 also {\cite[\S 3]{Gromov3}}]
\label{gradient}
 Let $\Gamma$ be a finitely generated group equipped with a
 $\Gamma$-invariant, finitely supported, symmetric random walk. Let $Y$ 
 be a $\cat$ space and $\rho \colon \Gamma \longrightarrow \isom(Y)$ a
 homomorphism. Suppose there is a positive constant $C$ such that
 $|-\Delta f(e)|^2 \geq C E_{\mu}(f)$ holds for every
 $\rho$-equivariant map $f$. Then there is a $\rho$-equivariant map
 $f_0$ such that $E_{\mu}(f_0)=0$. 
\end{Proposition}

Recall that $\mu$ is called {\it irreducible} if $\mu$ satisfies the 
following: for any $\gamma$, $\gamma'\in \Gamma$, there exists 
$\gamma_0,\gamma_1, \dots, \gamma_n \in \Gamma$ such that
$\gamma=\gamma_0$, $\gamma'=\gamma_n$, and
$\mu(\gamma_i,\gamma_{i+1})\not=0$, $i=0,\dots, n-1$. 
If $E_{\mu}(f_0)=0$, then $d(f(\gamma), f(\gamma'))=0$ for any 
pair $\gamma, \gamma'$ whenever there exists a sequence 
$\gamma_0, \gamma_1, \dots, \gamma_n \in \Gamma$ as above. 
Therefore, if $\mu$ is irreducible, then $f_0$ with $E_{\mu}(f_0)=0$ 
must be a constant map, and the image of $f_0$ is a fixed-point of the
action of $\rho(\Gamma)$. 

Let $\mu$ be the transition probability of the standard random walk on
$\Gamma$ with generating set 
$S=\{s_1, \dots, s_m, s_1^{-1}, \dots, s_m^{-1}\}$, namely, 
\begin{equation*}
 \mu(\gamma, \gamma')
  = \frac{\# \{s \in S \mid \gamma'=\gamma s\}}{2m}.
\end{equation*}
We note that $s \in S$ may become the identity element in $\Gamma$, and
if this is the case, then $\mu(\gamma,\gamma)\not= 0$ for any 
$\gamma \in \Gamma$. It is clear that so defined $\mu$ satisfies the
four properties listed above and the irreducibility. 

Denote by $\mu^n$ the $n$th convolution of $\mu$:
\begin{equation*}
 \mu^n(\gamma, \gamma')= \sum_{\gamma_1 \in \Gamma} 
 \dots \sum_{\gamma_{n-1} \in \Gamma} \mu(\gamma, \gamma_1) \dots
 \mu(\gamma_{n-1},\gamma'). 
\end{equation*}
We define {\it $n$-step energy} $E_{\mu^n}(f)$ of a $\rho$-equivariant
map $f$ by
\begin{equation*}
 E_{\mu^n}(f)= \frac{1}{2} 
 \sum_{\gamma' \in \Gamma}\mu^n(\gamma,\gamma')
 d_Y(f(\gamma),f(\gamma'))^2.
\end{equation*}
Since $\mu^n$ is $\Gamma$-invariant as $\mu$, again this does not depend
on the choice of $\gamma$.  

In \cite{Gromov3}, Gromov gave an interesting criterion for an action of
$\Gamma$  on $Y$ to have a fixed point in terms of the ratio of
$E_{\mu}(f)$ and $E_{\mu^n}(f)$ (see also \cite{Izeki-Kondo-Nayatani2}
and \cite{Silberman}).  We need a variant of this criterion, stated as
Theorem \ref{thm:n-step_ineq} below, in order to prove our main theorems. 

A slight modification of
the proof of \cite[Lemma 2.7(2)]{Izeki-Kondo-Nayatani2} gives the
following lemma, which we state in a way convenient for our use.  

\begin{Lemma}
 \label{lem:energy_decomposition}
 Let $\Gamma$ be a finitely generated group equipped with the standard
 random walk, and $Y$ a $\cat$ space. 
 Let $n$, $k$, $l$ be natural numbers satisfying $n \geq kl$.  Then, 
 for any natural number $i \leq 2k$, we have
 \begin{equation}
 \label{eq:energy_ineq}
  E_{\mu^{2n-il}}(f) \leq (2k-i+1)^2
   \left(E_{\mu^l}(f) + E_{\mu^{2n-2kl}}(f)\right)  
 \end{equation}
 and
 \begin{equation*}
  \left| -\Delta_{2n-il}f(e)\right|
   \leq (2k-i+1)
    \sqrt{2\left(E_{\mu^l}(f) + E_{\mu^{2n-2kl}}(f)\right)}, 
 \end{equation*}
where $-\Delta_j f(e)$ is the barycenter of 
$\{F_{f(e)}(\gamma)=\pi_{f(e)}\circ f(\gamma) \mid 
\mu^j (e,\gamma)\not= 0\}$ 
with respect to a measure 
$\sum_{\gamma}\mu^j(e,\gamma)\mathrm{Dirac}_{F_{f(e)}(\gamma)}$. 
\end{Lemma}

\begin{proof}
Noting $\mu^{2n-il}=(\mu^l)^{2k-i} \ast \mu^{2n-2kl}$, we get
\begin{equation}\label{computation3}
  \begin{split}
 & E_{\mu^{2n-il}}(f)  \\
  = & \frac{1}{2}\sum_{\gamma_1,\dots,\gamma_{2k-i+1}}
  \mu^l(e,\gamma_1) \cdots \mu^l(\gamma_{2k-i-1}, \gamma_{2k-i})
    \mu^{2n-2kl}(\gamma_{2k-i},\gamma_{2k-i+1}) \\
  & \times d_Y(f(e),f(\gamma_{2k-i+1}))^2 \\
  \leq & \frac{1}{2}\sum_{\gamma_1,\dots \gamma_{2k-i+1}}
   \mu^l(e,\gamma_1) \cdots
    \mu^l(\gamma_{2k-i-1}, \gamma_{2k-i})
    \mu^{2n-2kl}(\gamma_{2k-i},\gamma_{2k-i+1}) \\
  & \times (2k-i+1)(d_Y(f(e),f(\gamma_1))^2+ \dots + 
       d_Y(f(\gamma_{2k-i}),f(\gamma_{2k-i+1}))^2). 
  \end{split}
 \end{equation}
 Note that
 \begin{equation*}
  \begin{split}
   & \frac{1}{2}\sum_{\gamma_1, \dots, \gamma_{2k-i+1}}
    \mu^l(e,\gamma_1) \dots 
     \mu^{2n-2kl}(\gamma_{2k-i}, \gamma_{2k-i+1})
     d_Y(f(\gamma_{2k-i}),f(\gamma_{2k-i+1}))^2 \\
   =& \frac{1}{2}\sum_{\gamma_1, \dots, \gamma_{2k-1}}
      \mu^l(e,\gamma_1) \dots \mu^l(\gamma_{2k-i-1}, \gamma_{2k-i}) \\
    & \times \sum_{\gamma_{2k-i+1}}
     \mu^{2n-2kl}(\gamma_{2k-i}, \gamma_{2k-i+1})
     d_Y(f(\gamma_{2k-i}), f(\gamma_{2k-i+1}))^2 \\
   = & \sum_{\gamma_1, \dots, \gamma_{2k-i}}\mu^l(e,\gamma_1) \dots 
     \mu^l(\gamma_{2k-i-1},\gamma_{2k-i}) \times E_{\mu^{2n-2kl}}(f) \\
   = & E_{\mu^{2n-2kl}}(f), 
\end{split}
\end{equation*}
and, for $j< 2k-i$, 
 \begin{equation*}
  \begin{split}
   & \frac{1}{2}\sum_{\gamma_1, \dots, \gamma_{2k-i+1}}\mu^l(e,\gamma_1) \dots 
     \mu^{2n-kl}(\gamma_{2k-i}, \gamma_{2k-i+1})
     d_Y(f(\gamma_j),f(\gamma_{j+1}))^2 \\
   = & \frac{1}{2}\sum_{\gamma_1,\dots,\gamma_{2k-i}} 
     \mu^l(e,\gamma_1) \dots
     \mu^l(\gamma_j, \gamma_{j+1}) \dots 
     \mu^l(\gamma_{2k-i-1}, \gamma_{2k-i}) \\
     & \times \sum_{\gamma_{2k-i+1}}
     \mu^{2n-2kl}(\gamma_{2k-i}, \gamma_{2k-i+1}) 
     d_Y(f(e),f(\gamma_j^{-1}\gamma_{j+1}))^2 \\
   = & \frac{1}{2} \sum_{\gamma_1,\dots,\gamma_{2k-i-1}} 
     \mu^l(e,\gamma_1) \dots
     \mu^l(\gamma_j, \gamma_{j+1}) \dots 
     \mu^l(\gamma_{2k-i-2},\gamma_{2k-i-1})\\
     & \times \sum_{\gamma_{2k-i}} 
     \mu^l(\gamma_{2k-i-1}, \gamma_{2k-i}) 
     d_Y(f(\gamma_j), f(\gamma_{j+1}))^2 \\
   = & \frac{1}{2}\sum_{\gamma_1, \dots, \gamma_j} 
     \mu^l(e,\gamma_1) \dots
     \mu^l(\gamma_{j-1},\gamma_j) 
      \times\sum_{\gamma_{j+1}} 
      \mu^l(\gamma_j, \gamma_{j+1}) 
      d_Y(f(\gamma_j), f(\gamma_{j+1}))^2  \\  
   = & \sum_{\gamma_1} \dots \sum_{\gamma_{j}}
     \mu^l(e,\gamma_1) \dots
     \mu^l(\gamma_{j-1},\gamma_j)E_{\mu^l}(f) \\
   = & E_{\mu^l}(f)
  \end{split}
 \end{equation*}
 holds. 
 Together with \eqref{computation3}, this implies (\ref{eq:energy_ineq}):
\begin{equation*}
 \begin{split}
   E_{\mu^{2n-il}}(f) 
  & \leq (2k-i+1)\left((2k-i) E_{\mu^l}(f)+ E_{\mu^{2n-2kl}}(f)\right) \\
  & \leq (2k-i+1)^2 \left(E_{\mu^l}(f)+E_{\mu^{2n-2kl}}(f)\right). 
 \end{split}
\end{equation*}
 Since 
 $|-\Delta_{2n-il} f(e)| \leq \sqrt{2E_{\mu^{2n-il}}(f)}$ follows
 from \cite[Lemma 2.7(1)]{Izeki-Kondo-Nayatani2}, combining with 
 (\ref{eq:energy_ineq}) implies the other inequality. 
\end{proof}

The following is a
variant of \cite[Proposition 2.8]{Izeki-Kondo-Nayatani2}. 
\begin{Proposition}
 Let $\Gamma$ be a finitely generated group equipped with the standard
 random walk, and $Y$ a $\cat$ space. 
 If natural numbers $n$, $k$ and $l$ satisfy $n \geq kl$, we have
\begin{equation}
\label{eq:n-step_ineq}
 \begin{split}
 E_{\mu^{2n}}(f) \geq & E_{\mu^{2n-2kl}}(f) + 2k E_{\mu^l}(f) \\
   & - \sum_{i=1}^{2k} (2k-i+1)
       \sqrt{2\left(E_{\mu^l}(f)+E_{\mu^{2n-2kl}}(f)\right)}
   \left| -\Delta_l f(e)\right|.
 \end{split}
\end{equation}
\end{Proposition}

\begin{proof}
 By using \cite[Proposition 2.5]{Izeki-Kondo-Nayatani2} $2k$ times, we
 obtain 
\begin{equation*}
 \begin{split}
  E_{\mu^{2n}}(f) 
  & \geq E_{\mu^{2n-l}}(f) + E_{\mu^l}(f) 
  - \langle -\Delta_{2n-l} f(e), -\Delta_{l}f(e) \rangle \\
  & \geq \left(E_{\mu^{2n-2l}}(f) + E_{\mu^l}(f)
    - \langle -\Delta_{2n-2l} f(e), -\Delta_{l}f(e) \rangle \right)\\
  & \phantom{\geq} +E_{\mu^l}(f)- 
    \langle -\Delta_{2n-l} f(e), -\Delta_{l}f(e) \rangle \\
  & \vdots \\
  & \geq E_{\mu^{2n-2kl}}(f) + 2kE_{\mu^l}(f) 
  - \sum_{i=1}^{2k} \langle -\Delta_{2n-il} f(e), -\Delta_l f(e) \rangle \\
  & \geq E_{\mu^{2n-2kl}}(f) + 2kE_{\mu^l}(f) 
  - \sum_{i=1}^{2k} |-\Delta_{2n-il} f(e)|| -\Delta_l f(e)|. 
 \end{split}
\end{equation*}
Now Lemma \ref{lem:energy_decomposition} implies the desired
 inequality. 
\end{proof}
\begin{Theorem}
\label{thm:n-step_ineq}
 Let $\Gamma$ be a finitely generated group equipped with the standard
 random walk, and $Y$ a $\cat$ space. 
 Suppose that natural numbers $n$, $k$ and $l$ satisfy $n \geq kl$, and
 that there exists $\varepsilon>0$ such that, for any $\rho$-equivariant
 map $f$, the following inequality holds$:$
 \begin{equation*}
  E_{\mu^{2n}}(f) \leq (2k-\varepsilon) E_{\mu^l}(f). 
 \end{equation*}
 Then there is a $\rho$-equivariant map $f_0$ with $E_{\mu^l}(f_0)=0$. 
 Moreover,  $\rho(\Gamma)$ admits a fixed point. 
\end{Theorem}

\begin{proof}
 By (\ref{eq:n-step_ineq}) and the assumption, we get
\begin{equation*}
 \begin{split}
 (2k-\varepsilon)E_{\mu^l}(f)
 \geq & E_{\mu^{2n-2kl}}(f) + 2k E_{\mu^l}(f) \\
   & - \sum_{i=1}^{2k} (2k-i+1)
     \sqrt{2\left(E_{\mu^l}(f)+E_{\mu^{2n-2kl}}(f)\right)}
   \left| -\Delta_l f(e)\right|, 
 \end{split}
\end{equation*}
and hence
\begin{equation*}
 \begin{split}
 & c(k) \left|-\Delta_l f(e)\right| 
 \sqrt{E_{\mu^l}(f)+E_{\mu^{2n-2kl}}(f)} \\
 \geq & \varepsilon E_{\mu^l}(f) + E_{\mu^{2n-2kl}}(f) \\  
 \geq & \varepsilon\left(
    E_{\mu^l}(f) + E_{\mu^{2n-2kl}}(f)
   \right), 
 \end{split}
\end{equation*}
since we may assume $\varepsilon \leq 1$.  Hence
\begin{equation*}
 \begin{split}
  \left| -\Delta_l f(e)\right|
  & \geq \frac{\varepsilon}{c(k)} 
       \sqrt{E_{\mu^l}(f) + E_{\mu^{2n-2kl}}(f)} \\
  & \geq \frac{\varepsilon}{c(k)} 
         \sqrt{E_{\mu^l}(f)}. 
 \end{split}
\end{equation*}
By Proposition \ref{gradient}, this implies the existence of $f_0$ with
 $E_{\mu^l}(f_0)=0$. 
 If $l$ is odd, then $\mu^l(e,s)\not= 0$ for any $s \in S$.  Thus
 $f_0(s)=f_0(e)$ for any $s \in S$.  Since $f_0$ is $\rho$-equivariant,
 this implies that $f_0(\gamma)=f_0(e)$ for any $\gamma \in \Gamma$,
 which means that $f_0$ is a constant map and $f_0(e)$ is a fixed point. 
 Suppose $l$ is even. Let $\Gamma'$ be a subgroup of $\Gamma$ generated
 by $S'=\{\gamma \in \Gamma \mid \mu^2(e,\gamma)\not= 0\}$.  
 Then $\mu^l(e,\gamma)\not= 0$ for any $\gamma \in S'$, and $f_0(e)$ is
 fixed by $\rho(\Gamma')$. On the other hand, 
\begin{equation*}
 \varphi(\gamma)=
 \begin{cases}
  0 & \text{if }\gamma 
      \text{ can be expressed by word with even length,} \\
  1 & \text{otherwise}
 \end{cases}
\end{equation*}
 defines a homomorphism $\varphi:\Gamma \longrightarrow \Z/2\Z$, and the 
 kernel is exactly $\Gamma'$.  Therefore, $\Gamma'$ is a subgroup of
 index at most $2$. Since $f_0(\Gamma')$ is a point and the image 
 $f_0(\Gamma)$ is the $\rho(\Gamma)$-orbit of $f_0(\Gamma')$, 
 $f_0(\Gamma)$ consists of at most two points. In other words, 
 we have an orbit consisting of at most two points.  If the orbit
 consists of a single point, then the point is a fixed point for the
 action of $\rho(\Gamma)$.  Suppose the orbit consists of two points. 
 In this case, the action $\rho(\Gamma)$ is simply a permutation
 of these two points; if $\gamma \in \Gamma'$,
 then $\rho(\gamma)$ must fix both points, while if 
 $\gamma \not\in \Gamma'$, then $\rho(\gamma)$ exchanges these two
 points.  Since the action of $\rho(\Gamma)$ is by isometry, their
 unique midpoint must be fixed by $\rho(\Gamma)$.  (The uniqueness of
 the midpoint follows from the uniqueness of a geodesic joining given
 two points, which is true because $Y$ is a $\cat$ space.)
 This completes the proof.  
\end{proof}

According to Theorem \ref{thm:n-step_ineq}, in order to prove the
fixed-point property of $\Gamma$ for $Y$, we need to give a bound on the
ratio $E_{\mu^{2n}}(f)/E_{\mu^l}(f)$, which is less than $2k$ and
independent of $\rho$ and $f$. 
The ratio $E_{\mu^{2n}}(f)/E_{\mu^l}(f)$ has an interesting connection
with the Rayleigh quotient of a map from a graph into $Y$ as we explain
below. 

Let $\Gamma$ be the free group generated by $S$: 
$\Gamma=\langle s_1 \rangle \ast \dots \ast \langle s_m \rangle$. 
Let $\Gamma_P$ be the group given by a presentation 
$P=(S,R) \in P(m, \eta, d)$. Let $G_P=(V,E_P)$ be the graph associated
to $P$ defined in \S \ref{sec:graph-assoc-pres}. 
Let $Y$ be a $\cat$ space with a metric $d=d_Y$, and 
$f \colon V \longrightarrow Y$ a map into $Y$. 
Then the {\it Rayleigh quotient} 
$RQ^G(f)$ of $f$, introduced by Gromov, is defined as 
\begin{equation*}
 RQ^G(f) = \frac{\Vert df \Vert^2_{L^2}}{F(f)}, 
\end{equation*}
where
\begin{equation*}
 \begin{cases}
   \Vert df \Vert^2_{L^2} = 
    \frac{1}{2}\sum_{v \in V} \sum_{v' \in V}
    \nu(v,v') d_Y(f(v),f(v'))^2 \\
    F(f) = \frac{1}{2}\sum_{v,v' \in V} 
      \nu(v) \nu(v') d_Y(f(v),f(v'))^2.
 \end{cases}
\end{equation*}
When $Y=\R$ this is nothing but the Rayleigh quotient introduced in 
\S \ref{sec:graph-assoc-pres}. For a general $\cat$ space $Y$, using the
triangle inequality, one sees that
\begin{equation*}
\begin{split} 
 F(f) & \leq  \frac{1}{2} \sum_{v,v' \in V} \nu(v)\nu(v')
        \left(2d_Y(\mathrm{bar}(f),f(v))^2
            +2d_Y(\mathrm{bar}(f), f(v'))^2\right) \\
    & = 2 \sum_{v\in V} d_Y(\mathrm{bar}(f), f(v)), 
\end{split}
\end{equation*}
where $\mathrm{bar}(f)$ is the barycenter of $f(V)$.  Let $\delta(Y)$ be
the invariant introduced in \cite[Definition 6.1]{Izeki-Nayatani}, and
$\lambda_1(G,Y)$ be Wang's invariant (see \cite{Wang2} or
\cite[Definition 4.2]{Izeki-Nayatani}).
Then by \cite[Proposition 6.3]{Izeki-Nayatani},
for any nonconstant map $f \colon V \longrightarrow Y$, we obtain
\begin{equation}
 \label{eq:delta_and_rayleigh_quotient}
 \begin{split}
   RQ^G(f) & \geq \frac{1}{2}
   \frac{\Vert df \Vert^2_{L^2}}{\sum_{v\in V} d_Y(\mathrm{bar}(f), f(v))} \\
  & \geq \frac{1}{2} \lambda_1(G,Y)
   \geq \frac{1}{2}(1-\delta(Y)) \lambda_1(G).
 \end{split}
\end{equation}

Let $\rho \colon \Gamma_P \longrightarrow \isom(Y)$ be a homomorphism,
and take any $\rho$-equivariant map 
$f \colon \Gamma_P \longrightarrow Y$. 
Then we can consider $\tilde \rho$-equivariant map 
$\tilde f =f \circ \pi_P \colon \Gamma \longrightarrow Y$, where 
$\pi_P$ denotes the projection from $\Gamma$ onto $\Gamma_P$ and 
$\tilde \rho = \rho \circ \pi_P \colon \Gamma \longrightarrow \isom(Y)$. 
Regarding $v \in V$ as an element 
$\pi_P(\overline{v}) \in \Gamma_P$ in
a natural way, we get a map $\overline{f} \colon V \longrightarrow Y$
induced from $f$, where $\overline{v}$ denotes the element in $\Gamma$
determined by a word $v$. 

Regarding a $\rho$-equivariant map as a map defined on $V$ as above, 
we will be able to show that the Rayleigh quotient approximates the
ratio of $l$-step and $2n$-step energies with high probability. 
\begin{Proposition}
\label{prop:random_Rayleigh_quotient}
 Fix a natural number $k$, and let $k/(2k+1)<d<1/2$.  
 Set $l$ and $n$ as described in {\rm (\ref{eq:k_and_l})}. 
 Then, for any $\varepsilon>0$, we have
 \begin{equation*}
  \lim_{\eta\to \infty}
  \frac{
   \# \left\{P \in P(m,\eta,d)
      \left\vert \ 
      \begin{matrix} 
       (\ref{eq:approx_ratio})
       \text{ holds for } \forall Y,\ 
       \forall \rho \colon \Gamma_P \longrightarrow \isom(Y), \\
       \forall 
       \text{nonconstant }\rho \text{-equivariant map } f
      \end{matrix}
      \right. 
       \right\}}
       {\# P(m,\eta,d)}=1, 
 \end{equation*}
 where 
 \begin{equation}
 \label{eq:approx_ratio}
  (1-\varepsilon)RQ^G(\overline{f}) 
  \leq \frac{E_{\mu^{l}}(\tilde f)}{E_{\mu^{2n}}(\tilde f)}
  \leq (1+\varepsilon)RQ^G(\overline{f}), 
 \end{equation}
 and $Y$ denotes a $\cat$ space. 
\end{Proposition}

The proof will be given in \S \ref{proof_of_random_eigenvalue}. 

\begin{Remark}
\label{remark:d<1/2}
 We are assuming $d<1/2$ in order to regard 
 $P(m,\eta,d)$ as the probability space describing an experiment 
 given by choosing a word $c_0(2m)^{d\eta}$ times randomly from 
 $W_{\eta}$, where $c^{-1}\leq c_0 \leq c$. 
 However, the assumption $d<1/2$ is not essential as we will explain in 
 Remark \ref{remark:drop_d<1/2}. 
\end{Remark}

\section{Proofs of Theorem  \ref{thm:main_Hilbert} and \ref{thm:main}}
\label{sec:proof}

We first give the proof of Theorem \ref{thm:main}, assuming Proposition
\ref{prop:random_eigenvalue} and \ref{prop:random_Rayleigh_quotient} and
then point out a necessary change in order to show Theorem
\ref{thm:main_Hilbert}.  

Note that $1-\frac{1}{2}\log_{2m}(8m-4) < 1/2$ holds,  and that we have 
$RQ^G(f) \geq \lambda_1(G_P,Y)\geq \frac{1}{2}(1-\delta(Y))\lambda_1(G_P)
\geq \frac{1}{2}(1-\delta)\lambda_1(G_P)$
for any $Y \in \mathcal{Y}_{\delta}$ 
by (\ref{eq:delta_and_rayleigh_quotient}) and our assumption. 
By Propositions \ref{prop:random_eigenvalue} and
\ref{prop:random_Rayleigh_quotient}, we see that if 
$k/(2k+1)<d<1-\frac{1}{2}\log_{2m}(8m-4)$, with 
probability tending to $1$ as $\eta\to \infty$, $\Gamma_P$ satisfies 
\begin{equation}
\label{eq:energy_ratio1}
 \frac{E_{\mu^l}(\tilde f)}{E_{\mu^{2n}}(\tilde f)} \geq 
 \frac{1-\delta}{2}(1-\varepsilon)
\end{equation}
for any $Y \in \mathcal{Y}_{\delta}$, 
$\rho \colon \Gamma_P \longrightarrow \isom(Y)$, and 
$\tilde \rho$-equivariant map $\tilde f$, where 
$\tilde \rho=\rho \circ \pi_P \colon \Gamma \longrightarrow \isom(Y)$. 
Rewriting this, we see that, for any $\varepsilon'>0$, 
$P \in P(m,\eta,d)$ satisfies the following property with high
probability, that is, probability tending to $1$ as $\eta\to \infty$:  
For any $Y \in \mathcal{Y}_{\delta}$,  
$\rho \colon \Gamma_P \longrightarrow \isom (Y)$, and for any  
$\tilde \rho$-equivariant map 
$\tilde f \colon \Gamma \longrightarrow Y$,
\begin{equation*}
 E_{\mu^{2n}}(\tilde f) \leq \frac{2}{1-\delta}(1+\varepsilon')
  E_{\mu^l}(\tilde f)
\end{equation*}
holds. 
Since we have $k > 1/(1-\delta)$, 
by taking $\varepsilon'$ to be sufficiently small, we may assume
\begin{equation*}
 \frac{2}{1-\delta}(1+\varepsilon') \leq (2k-\varepsilon'') 
\end{equation*}
for some $\varepsilon''>0$.  Therefore, we obtain, for any
$\rho \colon \Gamma_P \longrightarrow \isom(Y)$ and 
$\tilde \rho$-equivariant map $\tilde f$, 
\begin{equation*}
 E_{\mu^{2n}}(\tilde f) \leq (2k-\varepsilon'') E_{\mu^l}(\tilde f)
\end{equation*}
and hence $\rho(\Gamma_P)=\tilde \rho(\Gamma)$ admits a fixed point by
Theorem \ref{thm:n-step_ineq}. This implies 
fixed-point property of $\Gamma_P$ with high probability.  
This completes the proof of Theorem \ref{thm:main}. 

If $Y=\mathcal{H}$, then we have $RQ^G(f)\geq \lambda_1(G_P)$.  Therefore
we have 
\begin{equation*}
 \frac{E_{\mu^l}(\tilde f)}{E_{\mu^{2n}}(\tilde f)} \geq 
 1-\varepsilon
\end{equation*}
instead of (\ref{eq:energy_ratio1}), and hence
\begin{equation*}
 E_{\mu^{2n}}(\tilde f) \leq (1+\varepsilon')
  E_{\mu^l}(\tilde f). 
\end{equation*}
Now we need $1+\varepsilon' \leq (2k-\varepsilon'')$ in order to apply
Theorem \ref{thm:n-step_ineq}, and it suffices to take $k=1$. 
Therefore taking $d> 1/(2+1)=1/3$ is sufficient, and this completes the
proof of Theorem \ref{thm:main_Hilbert}. 

\section{Preliminaries on probability theory}
\label{sec:some-facts-from} 

In this section, we collect some facts from probability theory, which
will be used in the proof of Propositions \ref{prop:random_eigenvalue}
and \ref{prop:random_Rayleigh_quotient}.

\subsection{Notation and some inequalities}

Let $X$ be a random variable defined on a probability space 
$(\Omega, P)$. We denote the probability of $X$ taking its value in a
set $A$ by $\mathbb{P}(X \in A)$: 
\begin{equation*}
  \mathbb{P}(X \in A) = 
  \int_{\{\omega \colon X(\omega) \in A\}} dP(\omega). 
\end{equation*}
Suppose $X$ is nonnegative and let $\mu$ be the mean (expected value 
$\mathbb{E}(X)$) of $X$. Then, for $t>0$, 
\begin{equation*}
 \mu = \int X \ dP \geq \int_{X \geq t\mu} X\ dP \geq
 \int_{X \geq t\mu} t\mu \ dP = \mathbb{P}(X \geq t\mu) t\mu. 
\end{equation*}
 Dividing the both ends by $t\mu$ yields {\it Markov's inequality}:
\begin{equation*}
 \mathbb{P}(X \geq t\mu) \leq \frac{1}{t}. 
\end{equation*}
Setting $\varepsilon = t\mu$ gives
\begin{equation}
\label{eq:Markov_ineq}
 \mathbb{P}(X \geq \varepsilon)  \leq \frac{\mu}{\varepsilon}. 
\end{equation}

If $X$ is a real valued random variable with mean $\mu$ and variance
$\sigma^2$.  Then, for any $d>0$, 
\begin{equation*}
 \sigma^2 = \mathbb{E}((X-\mu)^2) \geq
 \int_{|X-\mu|\geq d} (X-\mu)^2 \ dP \geq
 \mathbb{P}(|X-\mu|\geq d)d^2.
\end{equation*}
Dividing the both ends by $d^2$, we obtain {\it Chebyshev's inequality}:
\begin{equation*}
 \mathbb{P}(|X -\mu|\geq d) \leq \frac{\sigma^2}{d^2}. 
\end{equation*}
By taking $d=\varepsilon \mu$, we get
\begin{equation*}
 \mathbb{P}(|X -\mu|\geq \varepsilon\mu) 
 \leq \frac{\sigma^2}{\varepsilon^2\mu^2}. 
\end{equation*}

Let $X$ be a binomially distributed random variable with size $n$ and
parameter $p$; namely,  $X$ satisfies  
\begin{equation*}
 \mathbb{P}(X=x)=\,_nC_x p^x (1-p)^{n-x}, \quad x=0,1,\dots,n.
\end{equation*}
Then we say that $X$ has the binomial distribution $\mathrm{Bi}(n,p)$.  
Such an $X$ has mean $\mu=np$ and variance $\sigma^2=np(1-p)$. 
Thus Chebyshev's inequality yields
\begin{equation*}
 \mathbb{P}(|X-\mu| \geq \varepsilon\mu)
 \leq \frac{np(1-p)}{\varepsilon^2n^2p^2}
 = \frac{1-p}{\varepsilon^2 \mu} 
 \leq \frac{1}{\varepsilon^2 \mu}. 
\end{equation*}
In particular, we see that if $\{X_i\}$ is a sequence of binomially
distributed random variable with mean $\mu_i \to \infty$, then 
$\mathbb{P}(|X_i - \mu_i| \geq \varepsilon \mu_i)\to 0$.  
Actually, there are finer and more useful bounds called 
{\it Chernoff bounds}: 
\begin{equation*}
 \begin{split}
  \mathbb{P}(X \geq \mu + t)  & \leq
   \exp \left(-\frac{t^2}{2(\mu+t/3)}\right) \\
  \mathbb{P}(X \leq \mu - t)  & \leq
   \exp \left(-\frac{t^2}{2\mu}\right). 
 \end{split}
\end{equation*}
See, for example, \cite[\S~2.1]{Janson-Luczak-Rucinski} for proofs.
In particular, taking $t=\varepsilon \mu$ with 
$0 < \varepsilon < 1$, we obtain
\begin{equation}
\label{eq:chernoff_bounds} 
 \mathbb{P}(|X - \mu|> \varepsilon \mu) \leq 
 2 \exp \left(- \frac{\varepsilon^2 \mu}{3}\right). 
\end{equation}

\subsection{Moment generating function for binomially distributed random
  variables}
\label{sec:moment}

Let $X$ be a random variable on a probability space. Then 
\begin{equation*}
 \mathbb{E}(\exp (tX)) 
 = \mathbb{E}\left(\sum_{k=0}^{\infty} \frac{1}{k!}t^kX^k\right)
 = \sum_{k=0}^{\infty} \frac{1}{k!} t^k \mathbb{E}(X^k)
\end{equation*}
is called the moment generating function of $X$. 
We can compute {\it $\alpha$th moment} $\mathbb{E}(X^{\alpha})$ of 
$X$ by
\begin{equation*}
 \mathbb{E}(X^{\alpha})= 
 \left. \frac{d^{\alpha}}{dt^{\alpha}} 
  \mathbb{E}(\exp(tX)) \right\vert_{t=0}. 
\end{equation*}
Let $X$ be a random variable with distribution $\mathrm{Bi}(n,p)$.  
Then the mean of $X$ is $np$, and by computing the average, 
we get 
\begin{equation*}
 \begin{split}
 \mathbb{E}(\exp t(X-np)) 
 & = \sum_{x=0}^n \exp t(x-np) \,_nC_x p^x q^{n-x} \\
 & = \sum_{x=0}^n (\,_nC_x (e^t p)^x  q^{n-x})
        e^{-tnp} \\
 & = (e^tp + q)^n e^{-tnp}, 
 \end{split}
\end{equation*}
where $q=1-p$. A computation shows the following. 

\begin{Lemma}
 \label{lem:moment_gen_func} 
 Let $X$ be a random variable with distribution $\mathrm{Bi}(n,p)$, and
 set $M(t)= \mathbb{E}(\exp t(X-np))$. Then, $\alpha$th derivative 
 $M^{(\alpha)}(t)$ of $M(t)$ can be computed as
\begin{equation*}
  M^{(\alpha)}(t) = \sum_{\beta=0}^{\alpha}
 \,_{\alpha}C_{\beta}
 \left[ \left(e^t p + q\right)^n \right]^{(\beta)}
 \left[ e^{-tnp}\right]^{(\alpha-\beta)}, 
\end{equation*}
where $[\cdot]^{(a)}$ denotes the $a$th derivative of $[\cdot]$ and
 they turned out to be
\begin{equation*}
 \begin{split}
 &  \left[ \left(e^t p + q\right)^n \right]^{(\beta)} 
   = \sum_{j=1}^{\beta} c(\beta,j) 
     \frac{n!}{(n-j)!}
      \left(e^t p + q\right)^{n-j} e^{j t} p^{j}
    \quad (\beta > 0), \\
 &  \left[ e^{-tnp} \right]^{(\alpha-\beta)} = 
     (-np)^{\alpha-\beta}e^{-tnp}. 
 \end{split}
\end{equation*}
Here $c(k,j)$ is determined by $c(j,j)=1$, 
$c(j,0)=0$, and $c(k+1,j)=c(k,j-1)+j c(k,j)$. 
\end{Lemma}

We will use the following estimate in the proof of Proposition
\ref{prop:random_eigenvalue}. 

\begin{Lemma}
 \label{prop:moment_estimate}
 For a random variable $X$ with distribution $\mathrm{Bi}(n,p)$ and 
 with mean $\mu=np <1$, we have 
\begin{equation*}
 |\mathbb{E}((X-np)^{\alpha})| \leq c(\alpha) np.
\end{equation*}
\end{Lemma}

\begin{proof}
 Since $\mathbb{E}((X-np)^{\alpha})=M^{(\alpha)}(0)$ as we have seen above, 
 we need to estimate $M^{(\alpha)}(0)$.  
 By Lemma \ref{lem:moment_gen_func}, 
 and $np<1$, if $\beta>0$ 
\begin{equation*}
 \left|
 \left[\left(e^t p + q\right)^n \right]^{(\beta)} |_{t=0}\right|
 \leq \sum_{j=1}^{\beta} c(\beta,j) \frac{n!}{(n-j)!}p^j
 \leq \sum_{j=1}^{\beta} c(\beta,j) n^j p^j
 \leq \sum_{j=1}^{\beta} c(\beta,j) np, 
\end{equation*}
 and if $\alpha>\beta$
\begin{equation*}
 \left| \left[e^{-tnp}\right]^{(\alpha-\beta)}|_{t=0} \right|
 \leq (np)^{\alpha-\beta} \leq np. 
\end{equation*}
Thus we obtain the desired estimate. 
\end{proof}

\subsection{Remarks on our model}
\label{sec:remark-our-model}

In order to prove Propositions \ref{prop:random_eigenvalue} and
\ref{prop:random_Rayleigh_quotient}, it is convenient to fix the
cardinality of $R$, the set of relations, and consider
\begin{equation*}
 P(m,\eta,d,c_0)=
 \{P=(S,R) \in P(m,\eta,d) \mid \# R = c_0(2m)^{d\eta}\}, 
\end{equation*}
where $c^{-1}\leq c_0 \leq c$. 
Then the model $P(m,\eta,d,c_0)$ can be regarded as the set of subsets
of $W_{\eta}$ with cardinality $c_0(2m)^{d\eta}$ equipped with the
uniform counting measure. 
Of course, this is not the same as the probability space
$P'(m,\eta,d,c_0)$ which describes an experiment given by choosing a
word $c_0(2m)^{d\eta}$ times randomly from $W_{\eta}$; there may be
multiply chosen words and there is a natural action of the permutation
group given by changing the order of the choice.  
If we take $d<1/2$, then the
probability of an elements in $P'(m,\eta,d,c_0)$ to have multiply
chosen words tends to $0$ as $\eta \to \infty$ as the following
computation shows. Note that, on the subset $P'_0(m,\eta,d,c_0)$ of
$P'(m,\eta,d,c_0)$ consisting of elements without multiply chosen
words, the action of the permutation group mentioned above is
free. Thus, taking the quotient by the action of 
the permutation group, we can regard $P'_0(m,\eta,d,c_0)$ as 
$P(m,\eta,d,c_0)$.

If $X$ is a random variable with distribution $\mathrm{Bi}(n,p)$, then 
\begin{equation*}
 \mathbb{P}(X \geq 2)= 1- \mathbb{P}(X=0 \text{ or }1)
 =1 -\left(
  (1-p)^n + n(1-p)^{n-1}p
  \right). 
\end{equation*}
Since
\begin{equation*}
 \begin{split}
  & (1-p)^n + n(1-p)^{n-1}p \\
 = & \sum_{\alpha=0}^n \,_nC_{\alpha} 1^{n-\alpha}(-p)^{\alpha}
    + np \sum_{\alpha=0}^{n-1} \,_{n-1}C_{\alpha} 
      1^{n-\alpha}(-p)^{\alpha} \\
 \geq & 1 - np - n^2p^2 - \dots -n^np^n \\
     & \phantom{1}
     + np(1-(n-1)p - (n-1)^2p^2 - \dots - (n-1)^{n-1}p^{n-1}) \\
 \geq & 1 -2(np)^2 - \dots -2(np)^n, 
 \end{split}
\end{equation*}
we obtain
\begin{equation*}
 \mathbb{P}(X \geq 2) \leq 2\sum_{\alpha=2}^n (np)^{\alpha}. 
\end{equation*}
Now take any $w \in W_{\eta}$, and denote by $X_w(P)$ the number of
times that $w$ is chosen in an element $P \in P'(m,\eta,d,c_0)$.  
Then $X_w$ is a random
variable on $P'(m,\eta,d,c_0)$ with distribution 
$\mathrm{Bi}(c_0(2m)^{d\eta},(2m)^{-\eta})$. 
By the computation above, we have
\begin{equation*}
 \mathbb{P}(X_w \geq 2) \leq 2\sum_{\alpha=2}^{c_0(2m)^{d\eta}}
  (c_0(2m)^{d\eta} (2m)^{-\eta})^{\alpha}
\end{equation*}
Since $\# W_{\eta}=(2m)^{\eta}$, we see that
\begin{equation*}
\begin{split}
   \mathbb{P}(X_w \geq 2 \text{ for some }w \in W_{\eta}) 
 \leq & (2m)^{\eta} \times 2\sum_{\alpha=2}^{c_0(2m)^{d\eta}}
  c_0^{\alpha} (2m)^{\alpha d\eta} (2m)^{-\alpha \eta} \\
 = & 2 \sum_{\alpha=2}^{c_0(2m)^{d\eta}} 
    c_0^{\alpha} (2m)^{\alpha d\eta} (2m)^{-(\alpha-1)\eta}  \\
 = & 2 c_0^2(2m)^{(2d-1)\eta}  \sum_{\alpha=2}^{c_0(2m)^{d\eta}}
     c_0^{\alpha-2} (2m)^{(\alpha-2)(d-1)\eta} \\
 \leq & 2c_0^2(2m)^{(2d-1)\eta} \ \frac{1}{1-c_0(2m)^{-(1-d)\eta}}
\end{split}
\end{equation*}
holds for large $\eta$, since $c_0 (2m)^{-(1-d)\eta}<1$ when $\eta$ is
large enough.  By taking $d<1/2$ and suitable constant $c'$, we see that  
\begin{equation*}
 \mathbb{P}(X_w \geq 2 \text{ for some }w \in W_{\eta}) 
 \leq c'(2m)^{(2d-1)\eta}  \to 0
 \quad (\eta \to \infty). 
\end{equation*}
Note that $c'$ can be taken so that depending only on $c$
taken in order to define $P(m,\eta,d)$ but independent of $c_0$. 
In other words, the measure of $P'_0(m,\eta,d,c_0)$ in the probability
space $P'(m,\eta,d,c_0)$ uniformlly tends to $1$.  Hence, assuming
$d<1/2$,  we may identify $P(m,\eta,d,c_0)$ with $P'(m,\eta,d,c_0)$.  We
will use this identification in the proof of Proposition
\ref{prop:random_Rayleigh_quotient}.  

In order to prove Proposition \ref{prop:random_eigenvalue}, we need
another model $P''(m,\eta,d,c_0)$ of random groups. 
In this model $P''(m,\eta,d,c_0)$,  each word in $W_{\eta}$ is
independently chosen to be a relation of $P=(S,R)$ with probability
$(c_0(2m)^{d\eta})/(2m)^{\eta}$. Therefore any subset of $W_{\eta}$ can
occur as $R$ of $P=(S,R)$ in $P''(m,\eta,d,c_0)$, and the probability of
the occurrence is given by
\begin{equation*}
 \mathbb{P}(P=(S,R)) = 
 \left(\frac{c_0 (2m)^{d\eta}}{(2m)^{\eta}}\right)^{\# R}
 \left(1-\frac{c_0
  (2m)^{d\eta}}{(2m)^{\eta}}\right)^{(2m)^{\eta}-\# R}. 
\end{equation*}
This model corresponds to a model of random graphs often denoted by
$G(n,p)$ in the literature, where $n=(2m)^{\eta}$ and 
$p=(c_0(2m)^{d\eta})/(2m)^{\eta}$ in our case; while our original
model $P(m,\eta,d,c_0)$ corresponds to another model of random graphs
often denoted by $G(n,M)$, where $n=(2m)^{\eta}$ and $M=c_0(2m)^{d\eta}$
in our case. 
In our setting, a well-known relation between $G(n,p)$ and $G(n,M)$ can
be read as follows. 

\begin{Lemma}[cf.~{\cite[p.~38]{Bollobas}} and 
 {\cite[p.~17]{Janson-Luczak-Rucinski}}]
 \label{lem:relation_between_models}
 Let $\mathcal{Q}$ be any property of a group, and denote by 
 $\mathbb{P}_P(G_P \in \mathcal{Q})$ 
 $($resp.~$\mathbb{P}_{P''}(G_P \in \mathcal{Q})$\,$)$ the probability of 
 $G_P$ given by $P \in P(m,\eta,d,c_0)$ 
 $($resp.~$P \in P''(m,\eta,d,c_0)$\,$)$ having property $\mathcal{Q}$. 
 Then we have
 \begin{equation*}
  \mathbb{P}_{P}(G_P \in \mathcal{Q}) \leq
  3 \sqrt{c_0(2m)^{d\eta}}
  \mathbb{P}_{P''}(G_P \in \mathcal{Q}). 
 \end{equation*}
\end{Lemma}

\begin{proof}
Set $N=(2m)^{\eta}$,  $M=c_0(2m)^{d\eta}$, 
$p= (c_0(2m)^{d\eta})/(2m)^{\eta}$ and $q=1-p$. 
By the law of total probability, we have
\begin{equation*}
\begin{split}
 \mathbb{P}_{P''}(G_P \in \mathcal{Q}) 
 = &  \sum_{s=0}^{N} \mathbb{P}_{P,s}(G_P \in \mathcal{Q}) 
   \,_{N}C_s p^{s} q^{N-s} \\
 \geq & \mathbb{P}_{P,M}(G_P \in \mathcal{Q})
  \,_{N}C_M p^{M} q^{N-M}, 
\end{split}
\end{equation*}
where $P_{P,s}(G_P \in \mathcal{Q})$ denotes the probability of 
$G_P \in \mathcal{Q}$ in the model $P(m,\eta,d,c_0')$ 
with $s=c_0'(2m)^{d\eta}$. Now \cite[p.~4]{Bollobas} gives
\begin{equation*}
\begin{split}
  \,_{N}C_M p^M q^{N-M}
  \geq 
   e^{-1/(6M)} \frac{1}{\sqrt{2\pi}} (pqN)^{-1/2}.
\end{split}
\end{equation*}
Noting $M=pN$ implies the desired inequality. 
\end{proof}

\section{Proofs of Proposition \ref{prop:random_eigenvalue} and
 \ref{prop:random_Rayleigh_quotient}} 
\label{sec:conc-proofs-prop}

We first prove Proposition \ref{prop:random_Rayleigh_quotient} and then
proceed to Proposition \ref{prop:random_eigenvalue}. 

\subsection{Proof of Proposition \ref{prop:random_Rayleigh_quotient}}
\label{sec:proof-prop-RQ}

Let $c_0$ be a positive real number with $c^{-1} \leq c_0 \leq c$, where 
$c$ is the constant fixed in order to define $P(m,\eta,d)$, and set
\begin{equation*}
 P(m,\eta,d,c_0)=
 \{P=(S,R) \in P(m,\eta,d) \mid \# R = c_0(2m)^{d\eta}\}
\end{equation*}
as above. 
Recall that $\eta=2n+l \geq (2k+1)l$,  
take $d > k/(2k+1)$, and set $\varepsilon_0 = d(2k+1)-k$. 
Because of (\ref{eq:k_and_l}), we see that
\begin{equation}
\label{eq:d_and_n_and_ell}
 \begin{split}
  d\eta & \geq d(2k+1)l = \varepsilon_0 l + kl, \\
  d\eta -n & \geq d(2k+1)l - \left(kl + (2k+1)\right) 
          = \varepsilon_0 l - (2k+1),  \\
  d\eta -n & \leq d\left((2k+1)l + 2(2k+1)\right) - kl
          = \varepsilon_0 l + 2d(2k+1).
 \end{split}
\end{equation}
We will show that, for any $\varepsilon' >0$, there exist constants
$a_1$ and $a_2$ depending only on $m$, $k$, $c$, and 
$\varepsilon'$ 
(and independent of $l$ and $c_0$) such that  
\begin{equation*}
 \begin{split}
& \frac{\left\{
 P \in P(m,\eta,d,c_0) \left\vert
  \begin{matrix}
       (\ref{eq:approx_ratio})'
       \text{ holds for } \forall Y, 
       \forall \rho \colon \Gamma_P \longrightarrow \isom(Y), \\
       \forall 
       \text{nonconstant }\rho \text{-equivariant map } f
  \end{matrix} \right.
 \right\}}{\# P(m,\eta,d,c_0)}  \\
 \geq  & 1- a_1 \exp\left(-a_2(2m)^{\varepsilon_0 l}\right)
 \end{split}
\end{equation*}
holds for large $l$ (and hence for large $\eta$), where 
 \begin{equation*}
 \tag*{(\ref{eq:approx_ratio})$'$}
  (1-\varepsilon')RQ^G(\overline{f}) 
  \leq \frac{E_{\mu^{l}}(\tilde f)}{E_{\mu^{2n}}(\tilde f)}
  \leq (1+\varepsilon')RQ^G(\overline{f}). 
 \end{equation*} 
It is clear that this implies 
Proposition \ref{prop:random_Rayleigh_quotient}. 

Since we are assuming $d<1/2$, we may deal with
$P'(m,\eta,d,c_0)$ instead of $P(m,\eta,d,c_0)$ as we have just
explained in \S \ref{sec:remark-our-model}. 
Thus what we will actually show is that, for any
$\varepsilon' >0$, there exist constants $a_1$ and $a_2$ depending only
on $m$, $k$, $c$, and $\varepsilon'$ 
(and independent of $l$ and $c_0$) such that  
\begin{equation}
\label{eq:prop_rayleigh_reduced}
 \begin{split}
 &  \mathbb{P}\left(
  \begin{matrix}
       P \in P'(m,\eta,d,c_0) \text{ satisfies } 
       (\ref{eq:approx_ratio})' \\
       \text{ for } \forall Y, 
       \forall \rho \colon \Gamma_P \longrightarrow \isom(Y), \\
       \forall 
       \text{nonconstant }\rho \text{-equivariant map } f
  \end{matrix} \right) \\
 \geq  & 1- a_1 \exp\left(-a_2(2m)^{\varepsilon_0 l}\right)
 \end{split}
\end{equation}
holds for large $l$ (and hence for large $\eta$). 
Then this implies Proposition \ref{prop:random_Rayleigh_quotient}. 

Let $w \in W_{\eta}$. We denote the word given by the first $n$
letters of $w$ by $w_1$, the word formed by the next $l$ letters by
$u$, and the inverse word of the last $n$ letters by $w_2$; thus $w$
can be written as $w=w_1uw_2^{-1}$. 
Take $P=(S,R) \in P'(m,\eta,d,c_0)$. 
Let $\rho \colon \Gamma_P \longrightarrow \isom(Y)$ be a homomorphism,
and take any $\rho$-equivariant map 
$f \colon \Gamma_P \longrightarrow Y$. 
Then we can consider $\tilde \rho$-equivariant map 
$\tilde f \colon \Gamma \longrightarrow Y$, where 
$\tilde \rho = \rho \circ \pi_P \colon \Gamma \longrightarrow \isom(Y)$. 
Regarding $w_i$, $i=1,2$, as an element 
$\pi_P(\overline{w_i}) \in \Gamma_P$ in
a natural way, we get a map $\overline{f} \colon V \longrightarrow Y$
induced from $f$. 
Since $V = W_n$, we have
\begin{equation*}
 \begin{split}
 \Vert d \overline{f} \Vert^2_{L^2} 
 & = \frac{1}{2} \sum_{v, v' \in W_n} \nu(v,v') 
     d_Y(\overline{f}(v),\overline{f}(v'))^2  \\
 & = \frac{1}{2} \sum_{v, v' \in W_n} \nu(v,v') 
     d_Y(\tilde f(\overline{v}), \tilde f(\overline{v'}))^2  \\
 & = \frac{1}{2} \sum_{v,v' \in W_n} \nu(v,v')
     d_Y(\tilde f(e), \tilde f(\overline{v}^{-1}\overline{v'}))^2 \\
 & = \frac{1}{2} \sum_{\gamma \in \Gamma}
     \sum_{v,v' \in W_n \colon 
       \overline{v}^{-1}\overline{v'}=\gamma} \nu(v,v')
     d_Y(\tilde f(e), \tilde f(\gamma))^2 \\
 & = \frac{1}{2}\sum_{\gamma_P \in \Gamma_P} \ 
     \sum_{\gamma \in \Gamma \colon \pi_P(\gamma)=\gamma_P}\ 
     \sum_{v,v' \in W_n \colon 
       \overline{v}^{-1}\overline{v'}=\gamma} \nu(v,v')
     d_Y(f(e),  f(\gamma_P))^2. 
 \end{split}
\end{equation*}
Note that $w=w_1uw_2^{-1} \in R$ means that
$\pi_P(\overline{w_1u})=\pi_P(\overline{w_2})$ in $\Gamma_P$. Hence 
$\pi_P(\overline{w_1}^{-1}\overline{w_2})=\pi_P(\overline{u})$ and
$\pi_P(\overline{w_2}^{-1}\overline{w_1})=\pi_P(\overline{u}^{-1})$. 
Thus, recalling the definition of $\nu(v,v')$ and fixing 
$\gamma_P \in \Gamma_P$,  we can rewrite as 
\begin{equation*}
\begin{split}
   & \sum_{\gamma \in \Gamma \colon \pi_P(\gamma)=\gamma_P}\ 
       \sum_{v,v' \in W_n \colon 
       \overline{v}^{-1}\overline{v'}=\gamma }
      \nu(v,v') \\
 = & \sum_{\gamma \in \Gamma \colon \pi_P(\gamma)=\gamma_P}\ 
     \sum_{v,v' \in W_n \colon 
       \overline{v}^{-1}\overline{v'}=\gamma}
     \frac{\#\{w \in R \mid w_1=v \text{ and } w_2=v'\}}
      {2\#R} \\
   & \phantom{\sum_{v,v' \in W_n \colon 
       \overline{v}^{-1}\overline{v'}=\gamma}}
     + \frac{\#\{w \in R \mid w_1 = v' \text{ and } w_2=v\}}{2\# R} \\
 = & \frac{\#\{ w \in R \mid \pi_P(\overline{u})=\gamma_P
         \text{ or } 
         \pi_P(\overline{u}^{-1})=\gamma_P\}}{2\#R}. 
\end{split}
\end{equation*}
Therefore we get
\begin{equation}
 \label{eq:on_Gamma_P}
 \Vert d \overline{f} \Vert^2_{L^2}
 = \frac{1}{2} \sum_{\gamma_P \in \Gamma_P}
   \frac{\# \{w \in R \mid \pi_P(\overline{u})= \gamma_P
          \text{ or }\gamma_P^{-1} \}}{2\#R}
   d_Y(f(e),f(\gamma_P))^2, 
\end{equation}
Since
\begin{equation*}
 \# \{w \in R \mid \pi_P(\overline{u})= \gamma_P
          \text{ or }\gamma_P^{-1} \}
 = \sum_{\gamma \in \Gamma \colon \pi_P(\gamma)=\gamma_P}
 \# \{w \in R \mid \overline{u}= \gamma 
          \text{ or }\gamma^{-1}\}, 
\end{equation*}
lifting (\ref{eq:on_Gamma_P}) to $\Gamma$ yields
\begin{equation*}
 \Vert d \overline{f} \Vert^2_{L^2}
 = \frac{1}{2} \sum_{\gamma \in \Gamma}
   \frac{\# \{w \in R \mid \overline{u}= \gamma
          \text{ or }\gamma^{-1} \}}{2\#R}
   d_Y(\tilde f(e),\tilde f(\gamma))^2. 
\end{equation*}
Note that if $\overline{u}=e$, then 
$d_Y(\tilde f(e),\tilde f(\gamma))=d_Y(\tilde f(e), \tilde f(e))=0$, and
hence we may ignore this case in what follows. Then we have 
\begin{equation*}
 \begin{split}
 & \# \{w \in R \mid \overline{u}= \gamma
          \text{ or }\gamma^{-1} \}
   = \# \{w \in R \mid \overline{u}= \gamma\}
     + \# \{w \in R \mid \overline{u}=\gamma^{-1} \} \\
 = &  \sum_{v \in W_{l} \colon \overline{v}=\gamma}
        \# \{w \in R \mid u=v\}
      + \sum_{v' \in W_{l} \colon \overline{v'}=\gamma^{-1}}
        \# \{w \in R \mid u=v'\}
 \end{split}
\end{equation*}
for $\gamma \not= e$.  
Now let $X_{v,l} \colon P'(m,\eta,d,c_0) \longrightarrow \Z$ be a random 
variable defined by
\begin{equation*}
 X_{v,l}(P) = \# \{w \in R \mid u=v\}, 
\end{equation*}
where $v \in W_{l}$. 
By the definition of $P'(m,\eta,d,c_0)$, $X_{v,l}$ 
has the binomial distribution $\mathrm{Bi}(\#R,\frac{1}{\# W_l})$ with
mean $\frac{\# R}{\# W_l}$. 
Note that, for $P=(S,R) \in P'(m,\eta,d,c_0)$ under consideration, we
can write as  
\begin{equation}
\label{eq:df_with_random_variable}
\begin{split}
  \Vert d \overline{f} \Vert^2_{L^2}
 =  \frac{1}{2} \sum_{\gamma \in \Gamma}
    & \frac{\sum_{v \in W_{l} \colon \overline{v}=\gamma}
        X_{v,l}(P)
      + \sum_{v' \in W_{l} \colon \overline{v'}=\gamma^{-1}}
        X_{v',l}(P)}{2\#R} \\
    & \times d_Y(\tilde f(e),\tilde f(\gamma))^2. 
\end{split}
\end{equation}
Then, since 
$\mu^{l}(e,\gamma)= \#\{w \in W_{l}\mid \overline{w}=\gamma\}/\# W_{l}$,  
the expected value of the coefficient of 
$d_Y(\tilde f(e), \tilde f(\gamma))$ in
(\ref{eq:df_with_random_variable}) becomes 
\begin{equation*}
 \begin{split}
   &  \mathbb{E}\left(
   \frac{\sum_{v \in W_{l} \colon \overline{v}=\gamma}
        X_{v,l}
      + \sum_{v' \in W_{l} \colon \overline{v'}=\gamma^{-1}}
        X_{v',l}}{2\#R}
       \right)  \\
 = &  \frac{1}{2\# R}\left(
      \sum_{v \in W_{l} \colon \overline{v}=\gamma}
       \mathbb{E}(X_{v,l}) 
      + \sum_{v' \in W_{l} \colon \overline{v'}=\gamma^{-1}}
       \mathbb{E}(X_{v',l}) \right) \\
 = & \frac{1}{2\# R}\left(
      \sum_{v \in W_{l} \colon \overline{v}=\gamma}
       \frac{\# R}{\# W_{l}}
    + \sum_{v' \in W_{l} \colon \overline{v'}=\gamma^{-1}}
       \frac{\# R}{\# W_{l}} \right)\\
 = & \frac{1}{2} \left(\mu^{l}(e,\gamma) + \mu^{l}(e,\gamma^{-1})\right) \\
 = & \mu^{l}(e,\gamma). 
 \end{split}
\end{equation*}
Therefore, if $X_{v,l}(P)=\# \{w \in R \mid u=v\}$ of $P=(S,R)$ is close
to the expected value for all $v \in W_{l}$, then
$\Vert d \overline{f} \Vert^2_{L^2}$ must be close to  
\begin{equation*}
 \frac{1}{2} \sum_{\gamma \in \Gamma} \mu^{l}(e,\gamma)
  d_Y(\tilde f(e), \tilde f(\gamma))^2 = E_{\mu^{l}}(\tilde f).
\end{equation*}
More precisely speaking, if $P=(S,R)$ satisfies
\begin{equation*}
 (1-\varepsilon)\frac{\# R}{\# W_l} 
 \leq X_{v,l}(P) \leq
 (1+\varepsilon)\frac{\# R}{\# W_l}
\end{equation*}
for all $v \in W_l$, then we see that, for any $Y$, 
$\rho \colon \Gamma_P \longrightarrow \isom(Y)$ and 
$\rho$-equivariant map $f$, 
\begin{equation}
 \label{eq:df}
 (1-\varepsilon)E_{\mu^l}(\tilde f) 
 \leq \Vert d \overline{f} \Vert_{L^2}^2 \leq
 (1+\varepsilon)E_{\mu^l}(\tilde f)
\end{equation}
holds.

On the other hand, we have
\begin{equation*}
 \begin{split}
 F(\overline{f}) 
 & = \frac{1}{2} \sum_{v,v' \in W_n}\nu(v) \nu(v') 
   d_Y(\overline{f}(v),\overline{f}(v'))^2 \\
 & = \frac{1}{2} \sum_{v,v' \in W_n}\nu(v) \nu(v') 
   d_Y(f(\pi_P(\overline{v})), f(\pi_P(\overline{v'})))^2 \\
 & = \frac{1}{2} \sum_{v,v' \in W_n} \nu(v) \nu(v')
   d_Y(f(e), f(\pi_P(\overline{v}^{-1} \overline{v'})))^2 \\
 & = \frac{1}{2} \sum_{\gamma,\gamma' \in \Gamma}
   \sum_{v,v' \in W_n \colon 
     \overline{v}=\gamma, \overline{v'}=\gamma'} \nu(v) \nu(v')
   d_Y(\tilde f(e), \tilde f(\gamma^{-1} \gamma'))^2. 
 \end{split}
\end{equation*}
Let 
$X_{i,v,n} \colon P(m,\eta,d,c_0) \longrightarrow \Z$, $i=1,2$, be a  
random variable defined by
\begin{equation*}
 X_{i,v,n}(P) = \# \{w \in R \mid w_i = v\}, 
\end{equation*}
where $v \in W_n$. 
Using this random variable and recalling the definition of $\nu(v)$, we
see that, for $P=(S,R)$, 
\begin{equation*}
\begin{split}
  \nu(v) & = \sum_{v' \in W_n} \nu(v,v')
  = \frac{\#\{w \in R \mid w_1=v\} + \#\{w \in R \mid w_2=v\}}
    {2\# R} \\ 
  & = \frac{X_{1,v,n}(P)+X_{2,v,n}(P)}{2\# R}
\end{split}
\end{equation*}
holds.  Therefore we get 
\begin{equation*}
 \begin{split}
 F(\overline{f})
 & = \frac{1}{2} \sum_{\gamma,\gamma' \in \Gamma}
   \frac{\sum_{v \in W_n \colon \overline{v}=\gamma}
      \left(X_{1,v,n}(P) + 
      X_{2,v,n}(P)\right)}{2\# R}  \\
 & \phantom{\frac{1}{2} \sum_{\gamma,\gamma' \in \Gamma}}    
   \times
   \frac{\sum_{v \in W_n \colon \overline{v}=\gamma'}
      \left(X_{1,v,n}(P) + 
      X_{2,v,n}(P)\right)}{2\# R} \\
 & \phantom{\frac{1}{2} \sum_{\gamma,\gamma' \in \Gamma}} 
   \times 
   d_Y(\tilde f(e), \tilde f(\gamma^{-1} \gamma'))^2. \\
 \end{split}
\end{equation*}
Since $X_{i,v,n}$ has the binomial distribution 
$\mathrm{Bi}(\# R, \frac{1}{\# W_n})$ for $i=1,2$, we get
\begin{equation*}
 \begin{split}
    \mathbb{E}\left( \frac{\sum_{v \in W_n \colon \overline{v}=\gamma}
      X_{1,v,n} + X_{2,v,n}}{2\# R} \right)
 = & \frac{1}{2\# R} \sum_{v \in W_n \colon \overline{v}=\gamma} 
     \left(\mathbb{E}(X_{1,v,n}) 
            +\mathbb{E}(X_{2,v,n})\right) \\
 = & \frac{1}{\# R} \# \{v \in W_n \mid \overline{v}=\gamma\} 
          \frac{\# R}{\# W_n} \\
 = & \mu^n(e,\gamma). 
 \end{split}
\end{equation*}
Therefore, if both $\# \{w \in R \mid w_1 = v \}$ and 
$\# \{w \in R \mid w_2 = v \}$ are close to the expected
value for all $v \in W_n$, then $F(\overline{f})$ must be close to 
\begin{equation*}
 \begin{split}
 & \frac{1}{2} \sum_{\gamma, \gamma' \in \Gamma}
  \mu^n(e,\gamma) \mu^n(e,\gamma')
  d_Y(\tilde f(e), \tilde f(\gamma^{-1} \gamma'))^2 \\
 = & \frac{1}{2} \sum_{\gamma, \gamma' \in \Gamma}
  \mu^n(\gamma^{-1},e)\mu^n(\gamma^{-1},\gamma^{-1}\gamma')
   d_Y(\tilde f(e),\tilde f(\gamma^{-1}\gamma'))^2 \\
 = & \frac{1}{2} \sum_{\gamma} \mu^n(\gamma^{-1},e)
   \sum_{\gamma'} \mu^n(\gamma^{-1},\gamma^{-1}\gamma')
   d_Y(\tilde f(e), \tilde f(\gamma^{-1}\gamma'))^2 \\
 = & \frac{1}{2} \sum_{\gamma} \mu^n(e,\gamma^{-1})
   \sum_{\gamma''} \mu^n(\gamma^{-1},\gamma'') 
   d_Y(\tilde f(e), \tilde f(\gamma''))^2 \\
 = & \frac{1}{2} \sum_{\gamma''} 
   \sum_{\gamma} \mu^n(e,\gamma^{-1})
   \mu^n(\gamma^{-1},\gamma'') 
   d_Y(\tilde f(e), \tilde f(\gamma''))^2 \\
 = & \frac{1}{2} \sum_{\gamma''} \mu^{2n}(e,\gamma'')
   d_Y(\tilde f(e), \tilde f(\gamma''))^2 \\
 = & E_{\mu^{2n}}(\tilde f). 
 \end{split}
\end{equation*}
Actually, if $P=(S,R)$ satisfies 
\begin{equation*}
 \begin{split}
  & (1-\varepsilon)\frac{\# R}{\# W_n} 
   \leq X_{i,v,n}(P) \leq
    (1+\varepsilon)\frac{\# R}{\# W_n}, \quad i=1,2 
 \end{split}
\end{equation*}
for all $v \in W_n$, then we see that, for any $Y$, 
$\rho \colon \Gamma_P \longrightarrow \isom(Y)$, and $\rho$-equivariant
map $f$, 
\begin{equation}
 \label{eq:F(f)}
 (1-\varepsilon)^2 E_{\mu^{2n}}(\tilde f)
 \leq F(\overline{f}) \leq
 (1+\varepsilon)^2 E_{\mu^{2n}}(\tilde f)
\end{equation}
holds.

Since our random variables $X_{v,l}$ and $X_{i,v,n}$ are binomially
distributed with
distribution $\mathrm{Bi}(\#R, \frac{1}{\#W_l})$ and 
$\mathrm{Bi}(\#R, \frac{1}{\#W_n})$ respectively, we have 
\begin{equation*}
 \begin{split}
 \mathbb{E}(X_{v,l}) & = \frac{\# R}{\# W_l}
                       = c_0(2m)^{d\eta-l}
                       \geq c_0(2m)^{\varepsilon_0 l + (k-1)l} \\
 \mathbb{E}(X_{i,v,n}) & = \frac{\# R}{\# W_n}
                       = c_0(2m)^{d\eta-n} 
                       \geq c_0(2m)^{\varepsilon_0 l -(2k+1)}
 \end{split}
\end{equation*}
by (\ref{eq:k_and_l}) and (\ref{eq:d_and_n_and_ell}). 
Since each expected value does not depend on the choice of $v$, we
denote it by $\mu_l$ and $\mu_{i,n}$ respectively. 
Using (\ref{eq:chernoff_bounds}), we see that, for any given
$\varepsilon>0$,  
\begin{equation*}
 \begin{split}
  & \mathbb{P}(|X_{v,l}-\mu_l|>
    \varepsilon \mu_l ) 
   \leq 2\exp \left(
   -\frac{\varepsilon^2 c_0(2m)^{(\varepsilon_0 +k-1)l}}{3}
   \right), \\
  & \mathbb{P}(|X_{i,v,n}-\mu_{i,n}|>
    \varepsilon \mu_{i,n} ) 
   \leq 2\exp \left(
    -\frac{\varepsilon^2 c_0(2m)^{\varepsilon_0 l-(2k+1)}}{3}
    \right)  
 \end{split}
\end{equation*}
hold. 
Since $\# W_n=(2m)^n$, $n\leq kl + (2k+1)$, and 
$d(2k+1) = \varepsilon_0 + k$, we have
\begin{equation*}
 \begin{split}
  & \mathbb{P}(|X_{i,v,n}-\mu_{i,n}|>
    \varepsilon \mu_{i,n} \text{ for some }v\in W_n)  \\
  \leq
  & (2m)^n 2\exp \left(
   -\frac{\varepsilon^2 c_0 (2m)^{\varepsilon_0 l-(2k+1)}}{3}
   \right) \\
  \leq & 2\max_l \left\{(2m)^{kl+(2k+1)} \exp \left(
    -\frac{\varepsilon^2 c_0 (2m)^{\varepsilon_0 l-(2k+1)}}{6}
    \right) \right\} \\
   & \times  \exp \left( 
     -\frac{\varepsilon^2 c_0 (2m)^{\varepsilon_0 l-(2k+1)}}{6}
     \right). 
 \end{split}
\end{equation*}
Letting
\begin{equation*}
 \begin{split}
 a & = 2\max_l \left\{(2m)^{kl+(2k+1)} \exp \left(
    -\frac{\varepsilon^2 c^{-1} (2m)^{\varepsilon_0 l-(2k+1)}}{6}
    \right) \right\}, \\
 a' & =  \frac{\varepsilon^2 c^{-1}(2m)^{-(2k+1)}}{6}, 
 \end{split}
\end{equation*}
we obtain 
\begin{equation*}
    \mathbb{P}(|X_{i,v,n}-\mu_{i,n}|>
    \varepsilon \mu_{i,n} \text{ for some }v\in W_n)  \\
  \leq a \exp \left(-a' (2m)^{\varepsilon_0 l}\right)
\end{equation*}
for any $P'(m,\eta,d,c_0)$ for $c^{-1} \leq c_0 \leq c$. 
Since $k\geq 1$, 
\begin{equation*}
    \mathbb{P}(|X_{v,l}-\mu_{l}|>
    \varepsilon \mu_{l} \text{ for some }v\in W_{l})  \\
  \leq a \exp \left(-a' (2m)^{\varepsilon_0 l}\right)
\end{equation*}
holds for the same $a$ and $a'$ as above. 
Recalling (\ref{eq:df}) and (\ref{eq:F(f)}) and letting 
$a_1=3a$ and $a_2=a'$, we see that these
inequalities imply 
\begin{equation*}
  \mathbb{P}\left(
  \begin{matrix}
       P \in P'(m,\eta,d,c_0) \text{ satisfies } 
       (\ref{eq:approx_ratio2}) \\
       \text{ for } \forall Y, 
       \forall \rho \colon \Gamma_P \longrightarrow \isom(Y), \\
       \forall 
       \text{nonconstant }\rho \text{-equivariant map } f
  \end{matrix} \right) 
 \geq   1- a_1 \exp\left(-a_2(2m)^{\varepsilon_0 l}\right)
\end{equation*}
where 
\begin{equation}
 \label{eq:approx_ratio2}
  \frac{1-\varepsilon}{(1+\varepsilon)^2}RQ^G(\overline{f})
  \leq \frac{E_{\mu^{l}}(\tilde f)}{E_{\mu^{2n}}(\tilde f)} \leq
  \frac{1+\varepsilon}{(1-\varepsilon)^2}RQ^G(\overline{f}). 
\end{equation}
Thus, by taking $\varepsilon$ suitably, this implies 
(\ref{eq:prop_rayleigh_reduced}).  This completes 
the proof of Proposition \ref{prop:random_Rayleigh_quotient}.   

\begin{Remark}
\label{remark:drop_d<1/2}
 We should mention that it is also possible to prove Proposition
 \ref{prop:random_Rayleigh_quotient} using the model
 $P''(m,\eta,d,c_0)$ introduced in \S \ref{sec:remark-our-model} with
 the help of Lemma \ref{lem:relation_between_models}; 
 thus we can drop the assumption $d<1/2$ in 
 Proposition \ref{prop:random_Rayleigh_quotient}. 
 However, the appearance of $\mu^l$ and $\mu^n$ in the proof 
 becomes less natural compared with the proof presented above. 
\end{Remark}

\subsection{Proof of Proposition \ref{prop:random_eigenvalue}}
\label{proof_of_random_eigenvalue}
Our proof follows the idea of \cite{Chung-Lu-Vu}, although our model
is different from theirs and a certain modification is needed.   

As we have mentioned in \S \ref{sec:remark-our-model}, we will deal with
$P''(m,\eta,d,c_0)$ in order to prove Proposition
\ref{prop:random_eigenvalue}. 
Let $P \in P''(m, \eta, d, c_0)$ and $G_P=(V,E_P)$ the graph associated to
$P$.  Recall that $G_P$ has measures on $V \times V$ and $V$ both
denoted by $\nu$, which are used to define the inner product on the
space of functions on $V$ and the Laplacian $\Delta_P$. 
We are going to deal with the eigenvalues of $\Delta_P$, and it is 
convenient to give an expression of $\Delta_P$ by matrices. 
Let us denote $V=\{v_1, \dots, v_{(2m)^n}\}$, and
$\mathbf{e}_i$ be a function defined by
\begin{equation*}
 \mathbf{e}_i(v_j)=
 \begin{cases}
  \frac{1}{\sqrt{\nu(v_i)}} & \text{if }i=j, \\
  0                         & \text{if }i\not= j. 
 \end{cases}
\end{equation*}
Then $\{\mathbf{e}_1, \dots, \mathbf{e}_{(2m)^n}\}$ forms an orthonormal basis
of the space of functions on $V$ with respect to the inner product
$\left(\cdot, \cdot\right)_{L^2}$.  Note that
\begin{equation*}
 \Delta_P \mathbf{e}_i (v_j) = \mathbf{e}_i(v_j)
 - \sum_{v' \in V} \frac{\nu(v_j,v')}{\nu(v_j)}\mathbf{e}_i(v')
 = \frac{\delta_{ij}}{\sqrt{\nu(v_i)}} - 
    \frac{\nu(v_j,v_i)}{\nu(v_j)\sqrt{\nu(v_i)}}, 
\end{equation*}
therefore
\begin{equation*}
 \Delta_P \mathbf{e}_i = \mathbf{e}_i - 
 \sum_j \frac{\nu(v_j,v_i)}{\sqrt{\nu(v_j)\nu(v_i)}}
 \mathbf{e}_j,  
\end{equation*}
and $\Delta_P$ can be expressed as
\begin{equation*}
 \Delta_P = I - A', 
\end{equation*}
where $I$ denotes the identity matrix, and $A'$ is expressed by means of
its entries as 
\begin{equation*}
  A'=\left(\frac{\nu(v_i,v_j)}{\sqrt{\nu(v_i)\nu(v_j)}}\right). 
\end{equation*}
Note that, by definition, 
\begin{equation*}
 \frac{\nu(v_i,v_j)}{\sqrt{\nu(v_i)\nu(v_j)}}
 = \frac{\text{multiplicity of }\{v_i, v_j\}}
    {\sqrt{\mathrm{deg}(v_i)\mathrm{deg}(v_j)}}. 
\end{equation*}
Let $A$ be the multiplicity matrix (the adjacency matrix taking 
the multipicity in account) and $D$ the square root of the inverse of
the degree matrix:
\begin{equation*}
 A=\left(\text{multiplicity of }\{v_i,v_j\}\right)  \quad 
 \text{and} \quad
 D=\left(\frac{\delta_{ij}}{\sqrt{\mathrm{deg}(v_i)}}\right).   
\end{equation*}
Then we see that $A'= DAD$. 

Let $\pi_0$ be the orthogonal projection onto the space of constant
functions on $V$.  Note that 
\begin{equation*}
 M = (I-\Delta_P) (I-\pi_0)
 = I -\Delta_P - \pi_0 = DAD - \pi_0
\end{equation*}
has eigenvalues $0$ and $1 - \lambda_i(G_P)$, $i\not= 0$, where
$0=\lambda_0(G_P)\leq \lambda_1(G_P) \leq \dots \lambda_m(G_P)$ are the 
eigenvalues of 
$\Delta_P$, and eigenfunctions corresponding to $0$ are
constant functions.  
(Here we regard $M$ is given as a matrix form with respect the
orthonormal basis $\{\mathbf{e}_i\}$. Note that the eigenfunctions of
$M$ are the same as that of $\Delta_P$.) 
Thus we have
\begin{equation}
 \label{eq:eigenvalue_estimate}
 \max_{i\not= 0} |1-\lambda_i(G_P)| \leq \Vert M \Vert, 
\end{equation}
where $\Vert M \Vert$ denotes the operator norm of $M$.

We regard $\Vert M \Vert=\Vert M_P \Vert$ as a random variable 
on $P''(m,\eta,d,c_0)$
and will show that, for any $\varepsilon'>0$, there exist positive
constants $a_1$ and $a_2$ depending only on $m$, $k$,
$c$, and $\varepsilon'$ (and, in particular, independent of $c_0$)
and $a_3$ depending on $m$, $k$, $\varepsilon'$, and $\alpha$
such that
\begin{equation}
 \label{eq:approx_eigenvalue}
 \mathbb{P}_{P''}(\Vert M \Vert > \varepsilon') 
 \leq   a_1 \exp\left(-a_2(2m)^{\varepsilon_0 l}\right)
        + a_3(2m)^{-(\alpha \varepsilon_0 - k) l}
\end{equation}
holds, where $\alpha$ is an arbitrary natural number and 
$\varepsilon_0=d(2k+1)-k$.  
Once (\ref{eq:approx_eigenvalue}) is proven, recalling Lemma 
\ref{lem:relation_between_models} and 
$d\eta \leq (k+\varepsilon_0)l +2(k+\varepsilon_0)$,  
we obtain
\begin{equation*}
\begin{split}
 & \mathbb{P}_{P}(\Vert M \Vert > \varepsilon') \\
 \leq & 3 \sqrt{c_0(2m)^{d\eta}} 
      \mathbb{P}_{P''}(\Vert M \Vert > \varepsilon') \\
 \leq & 3 c^{1/2} (2m)^{(k+\varepsilon_0)l + 2(k+\varepsilon_0)} 
 \left(a_1 \exp\left(-a_2(2m)^{\varepsilon_0 l}\right)
        + a_3(2m)^{-(\alpha \varepsilon_0 -k) l} \right)\\
 \leq & a_1' (2m)^{(k+\varepsilon_0)l} 
     \exp\left(-a_2(2m)^{\varepsilon_0 l}\right)
     + a_3'(2m)^{-(\alpha \varepsilon_0 - (2k+\varepsilon_0))l}. 
\end{split}
\end{equation*}
By choosing $\alpha > (2k+\varepsilon_0)/\varepsilon_0$
and setting $\alpha_0 = \alpha - (2k+\varepsilon_0)/\varepsilon_0$, we
see that
\begin{equation*}
 \mathbb{P}_P (\Vert M \Vert \leq \varepsilon')  \geq
 1 -  a_1' (2m)^{(k+\varepsilon_0)l} \exp\left(-a_2(2m)^{\varepsilon_0 l}\right)
   - a_3''(2m)^{-(\alpha_0 \varepsilon_0)l}
 \to 1
\end{equation*}
as $\eta \to \infty$, where $a_1'$, $a_3''$, and $\alpha_0$ depends only
on $m$, $k$, $c$, $\varepsilon'$, and $\alpha$ chosen as above (and
independent of $c_0$). 
Together with (\ref{eq:eigenvalue_estimate}),
this  implies Proposition \ref{prop:random_eigenvalue}. 
In what follows we will deal only with $P''(m,\eta,d,c_0)$, and will
drop subscript $P''$ of $\mathbb{P}_{P''}$, since there will be no
chance of confusion. 

The first step of the proof of (\ref{eq:approx_eigenvalue}) is to
estimate $\Vert B \Vert$ of an operator $B$ defined below, which
approximates $M$ with high probability. 
Recall our random graph has $\#V=(2m)^{n}$ vertices and, 
for each pair of vertices, the multiplicity of the edge joining them is
at most $(2m)^l=\# W_l$.  Since each word in $W_{\eta}$ is chosen with
probability $(c_0(2m)^{d\eta})/(2m)^{\eta}$, 
the multiplicity of an edge has the binomial distribution 
$\mathrm{Bi}((2m)^l, (c_0(2m)^{d\eta})/(2m)^{\eta})$, 
and the expected multiplicity $\mu_e$ of each 
edge is $(2m)^{l} \times (c_0(2m)^{d\eta})/(2m)^{\eta}$.  
Thus, by recalling (\ref{eq:k_and_l}) and (\ref{eq:d_and_n_and_ell}), 
the expected multiplicity $\mu_e$ can be estimated as  
\begin{equation*}
 \begin{split}
 \mu_e & =  c_0(2m)^{l+(d-1)(2n+l)}  \\
 & \leq 2c_0(2m)^{d(2k+1)l + 2d(2k+1) - (2k+1)l+l} \\
 & = 2c_0(2m)^{-(k-\varepsilon_0)l + 2d(2k+1)}.   
 \end{split}
\end{equation*}
On the other hand, since a vertex is given as the first (or the last)
$n$ letters of elements of $W_{\eta}$, it is followed by (or follows)
the elements of $W_{n+l}$. 
Therefore, for any $v \in W_{n}$, the random variable $Y_{i,v,n}$ on
$P''(m,\eta,d,c_0)$ defined by
\begin{equation*}
 Y_{i,v,n}(P)= \# \{w \in R \mid w_i=v\}
\end{equation*}
has the binomial distribution 
$\mathrm{Bi}((2m)^{n+l}, (c_0 (2m)^{d\eta})/(2m)^{\eta})$.  Since the
degree of a vertex $v$ is given by
$\mathrm{deg}(v)=Y_{1,v,n}+Y_{2,v,n}$,  the expected degree $\mu_v$ of a
vertex is given by 
$\mu_v = 2 (2m)^{n+l} \times (c_0(2m)^{d\eta})/(2m)^{\eta}$. Thus, 
again by (\ref{eq:k_and_l}) and (\ref{eq:d_and_n_and_ell}), $\mu_v$ can
be estimated as
\begin{equation*}
\mu_v = 2 c_0(2m)^{d\eta}(2m)^{-n} \geq
2c_0(2m)^{\varepsilon_0 l - (2k+1)}.  
\end{equation*}

Let a matrix $B$ be 
\begin{equation*}
 B= \frac{1}{\mu_v} 
    \left(A - \mu_eK\right), 
\end{equation*}
where $K$ is a matrix all of whose entries are $1$. Note that each
entries of the matrix $A-\mu_e K$ is equal to 
\begin{equation*}
 (\text{multiplicity of } \{v_i,v_j\} )
 - \mu_e, 
\end{equation*}
and hence, writing $B=(b_{ij})$, we see that $\mathbb{E}(b_{ij})=0$ for
each $(i,j)$. 
We first show that the eigenvalue of $B$ approaches to $0$ as 
$l \to \infty$. Note that, since $B$ is symmetric
\begin{equation*}
 \mathrm{trace}(B^{2\alpha}) = 
 \lambda_1(B)^{2\alpha} + \dots \lambda_{(2m)^n}(B)^{2\alpha}
\end{equation*}
for any positive integer $\alpha$, 
where $\lambda_i(B)$'s  are eigenvalues of $B$ (counted with
multiplicity).  Let $\lambda_1(B)$ be the eigenvalue
of $B$ with maximum absolute value, which is equal to the operator norm
of $B$: $\Vert B \Vert =|\lambda_1(B)|$. 
Then  we have
\begin{equation*}
 \mathbb{E}(\Vert B \Vert^{2\alpha})=
 \mathbb{E}(\lambda_1(B)^{2\alpha}) \leq 
 \mathbb{E}(\mathrm{trace}(B^{2\alpha})). 
\end{equation*}
The $i$th diagonal entry of $B^{2\alpha}$ is a sum of terms of the form
$b_{ii_1}b_{i_1i_2} \dots b_{i_{2\alpha-1}i}$, which corresponds to a
closed path of length $2\alpha$ in the complete graph (with self-loops)
having $V$ as the vertex set with basepoint $v_i$.  

Note that $b_{ij}$'s are independent random variables, and that the mean
of $b_{ij}$ is $0$ as we have mentioned above. Therefore the expected
value of the diagonal entries $B^{2\alpha}$ involves only 
closed paths of length $2\alpha$ all of whose edges are visited at least
twice.  Such a closed path is called a {\it closed good walk} in
\cite{Furedi-Komlos}.  
Now we see that $\mathbb{E}(\mathrm{trace}(B^{2\alpha}))$ is a sum of
terms of the form
\begin{equation*}
 \mathbb{E}\left(\prod b_{ij}^{\alpha_{ij}}\right) = 
 \prod \mathbb{E}(b_{ij}^{\alpha_{ij}}), 
\end{equation*}
where $\alpha_{ij}$'s are positive integers greater than $1$ and their
sum is equal to $2\alpha$.  
The multiplicity of each edge is distributed binomially with mean
$\mu_e$ and, by taking $l$ suitably large, we may assume
$\mu_e \leq 2c_0(2m)^{-(k-\varepsilon_0)l+2d(2k+1)}<1$.  
Therefore, we can apply Lemma \ref{prop:moment_estimate} to obtain 
\begin{equation*}
 \mathbb{E}(b_{ij}^{\alpha_{ij}})
  = \frac{1}{\mu_v^{\alpha_{ij}}}
   \mathbb{E}\left(\left(
      \text{multiplicity of }\{v_i,v_j\}
     -\mu_e\right) ^{\alpha_{ij}} \right)\\
  \leq c(\alpha_{ij}) \frac{1}{\mu_v^{\alpha_{ij}}} \mu_e. 
\end{equation*}
Let $W(\alpha, \beta)$ be the set of all closed good walks of length
$2\alpha$ which pass through $\beta+1$ distinct vertices. 
If $b_{ii_1}b_{i_1i_2} \dots b_{i_{2\alpha-1}i}$ corresponds to a walk
in $W(\alpha, \beta)$, then there appear at least $\beta$ distinct
$b_{ij}$'s.  Since we may assume $\mu_e<1$, we get 
\begin{equation}
\label{eq:for_a_walk}
 \begin{split}
 \prod
 \mathbb{E}(b_{ij}^{\alpha_{ij}})
 & \leq c(\alpha) \frac{1}{\mu_v^{2\alpha}} \mu_e^{\beta} \\
 & \leq 2^{-2\alpha}c(\alpha) c_0^{\beta-2\alpha}
   (2m)^{-2\alpha \varepsilon_0 l + 2\alpha(2k+1)}
      (2m)^{-\beta (k-\varepsilon_0)l+2d\beta(2k+1) }
 \end{split}
\end{equation}
for some constant $c(\alpha)$. 
Note that 
\begin{equation*}
\begin{split}
 & \# W(\alpha,\beta) \\
 \leq &
 \# (\text{choice of }\beta+1 \text{ vertices from }V) \\
 & \times \# (\text{good walks with length }2\alpha
   \text{ in the complete graph with }
     \beta + 1 \text{ vertices}) \\
 \leq & \#V^{\beta+1} c(\alpha,\beta)
 = (2m)^{(\beta+1)n}c(\alpha, \beta) 
 \leq c(\alpha,\beta)(2m)^{(\beta+1)(kl + (2k+1))}
 \end{split}
\end{equation*}
holds, where, clearly,  $c(\alpha,\beta)$ is a constant depending only
on $\alpha$ and $\beta$. 
(See \cite{Furedi-Komlos} for a sharp bound of $W(\alpha,\beta)$.) 

Thus, noting $c^{-1}\leq c_0 \leq c$ and summing up
(\ref{eq:for_a_walk}) over all good walks of length $2\alpha$ gives 
\begin{equation*}
\begin{split}
 & \mathbb{E}(\mathrm{trace}(B^{2\alpha})) \\
 & \leq 2^{-2\alpha} c(\alpha) 
    \sum_{\beta = 0}^{\alpha} 
    c^{2\alpha-\beta}
 \# W(\alpha,\beta) (2m)^{-2\alpha \varepsilon_0 l}
    (2m)^{-\beta (k-\varepsilon_0)l} 
     (2m)^{(2\alpha + 2d\beta)(2k+1)}\\
 & \leq \left(\frac{c}{2}\right)^{2\alpha}c(\alpha)  
   \sum_{\beta=0}^{\alpha} c(\alpha,\beta)
   (2m)^{(\beta+1)(kl+(2k+1))}  (2m)^{(2\alpha+2d\beta)(2k+1)} \\
  & \phantom{\leq \left(\frac{c}{2}\right)^{2\alpha}c(\alpha)  
   \sum_{\beta=0}^{\alpha}}
   \times (2m)^{-2\alpha \varepsilon_0 l-\beta(k-\varepsilon_0)l} \\
 & \leq  \left(\frac{c}{2}\right)^{2\alpha} c(\alpha)
     \sum_{\beta =0}^{\alpha}c(\alpha,\beta) 
    (2m)^{(\beta+1 +2\alpha+2d\beta)(2k+1)}
    (2m)^{(\beta+1)kl-2\alpha\varepsilon_0 l-\beta(k-\varepsilon_0) l} \\
 & \leq \left(\frac{c}{2}\right)^{2\alpha} c(\alpha)(2m)^{(5\alpha+1)(2k+1)}
     \sum_{\beta=0}^{\alpha} c(\alpha,\beta) 
    (2m)^{kl-(2\alpha - \beta)\varepsilon_0 l}, 
\end{split}
\end{equation*}
where we have used $\beta \leq \alpha$ to deduce the last inequality. 
By $\beta \leq \alpha$ again, and by setting 
$c'(\alpha)=2^{-2\alpha} c(\alpha)c^{2\alpha}(2m)^{(5\alpha+1)(2k+1)}
\sum_{\beta} c(\alpha,\beta)$, 
we get 
\begin{equation*}
  \mathbb{E}(\mathrm{trace}(B^{2\alpha}))
  \leq c'(\alpha) (2m)^{kl - \alpha \varepsilon_0 l}. 
\end{equation*}
Thus, by Markov's inequality (\ref{eq:Markov_ineq}), 
for any $\varepsilon>0$, we get
\begin{equation*}
 \mathbb{P}(\Vert B \Vert^{2\alpha} \geq \varepsilon^{2\alpha}) 
  \leq \frac{\mathbb{E}(\Vert B \Vert^{2\alpha})}{\varepsilon^{2\alpha}}
  \leq \frac{\mathbb{E}(\mathrm{trace}(B^{2\alpha}))}
             {\varepsilon^{2\alpha}}
  =\frac{c'(\alpha)}{\varepsilon^{2\alpha}}
   (2m)^{kl -\alpha \varepsilon_0 l}, 
\end{equation*}
and hence
\begin{equation}
 \label{eq:prob_|B|}
 \mathbb{P}(\Vert B \Vert \geq \varepsilon) 
  \leq \frac{c'(\alpha)}{\varepsilon^{2\alpha}}
   (2m)^{-(\alpha \varepsilon_0-k) l}. 
\end{equation}

Next we will show that $B$ actually approximates $M$ with high
probability. Now take a matrix $M'$ to be
\begin{equation*}
 M' = D \left(A - \mu_e K\right) D. 
\end{equation*}
Writing $M'=(m_{ij}')$, we see that
\begin{equation*}
 m_{ij}'= \frac{\mu_v}
      {\sqrt{\mathrm{deg}(v_i)\mathrm{deg}(v_j)}} b_{ij}. 
\end{equation*}
Let $f \colon V \longrightarrow \R$ and write 
$f=\sum_i f_i \mathbf{e}_i$.  
Since $M'$ is symmetric, 
\begin{equation*}
 \begin{split}
  \Vert M' \Vert 
  & = \max_{\vert f \vert=1}
     \left| \left( f, M' f\right) \right| \\
  & = \max_{\vert f \vert = 1}
     \left| \sum_i f_i \left(
      \sum_j \frac{\mu_v}
                   {\sqrt{\mathrm{deg}(v_i)\mathrm{deg}(v_j)}}
         b_{ij}f_j \right) \right| \\
  & = \max_{\vert f \vert=1}
      \left| \sum_i \frac{\sqrt{\mu_v}}
                   {\sqrt{\mathrm{deg}(v_i)}} f_i
         \left(
      \sum_j b_{ij}\frac{\sqrt{\mu_v}}
                   {\sqrt{\mathrm{deg}(v_j)}} f_j \right) \right| \\ 
  & = \max_{\vert f \vert=1} \left|\left(f', Bf' \right) \right| \\
  & \leq \max_{\vert f \vert=1} \Vert B \Vert \vert f' \vert^2, 
 \end{split}
\end{equation*}
where $f'=\sum_i f_i' \mathbf{e}_i$  is defined by
\begin{equation*}
 f_i' = \frac{\sqrt{\mu_v}}
                   {\sqrt{\mathrm{deg}(v_i)}} f_i. 
\end{equation*}
By $\mathrm{deg}(v)=Y_{1,v,n}+Y_{2,v,n}$ and (\ref{eq:chernoff_bounds}),
we know that, for any $\varepsilon>0$, there is a constant $a$, and 
\begin{equation}
 \label{eq:deg_conc}
 \begin{split}
 & \mathbb{P}(|\mathrm{deg}(v) - \mu_v| > 
   \varepsilon \mu_v \text{ for some }v \in V) \\
 < & 2\# V \times 2 \exp \left(-\frac{\varepsilon^2 \mu_v}{3}\right) \\
 = & 4(2m)^n \exp \left(
   -\frac{\varepsilon^2 2c_0 (2m)^{\varepsilon_0 l-(2k+1)}}{3}
   \right) \\
 \leq & 2a \exp \left(-a' (2m)^{\varepsilon_0 l}\right), 
 \end{split}
\end{equation}
where $a$, $a'$ can be taken as the same constants as the preceding
subsection. 
Thus we have, for any given $\varepsilon>0$, 
\begin{equation}
\label{eq:M'_and_B}
   \Vert M' \Vert \leq \Vert B \Vert (1- \varepsilon)^{-1}
\end{equation}
with probability $1-2a \exp \left(-a' (2m)^{\varepsilon_0 l}\right)$,
which is independent of $c_0$ (but depends on $m$, $k$, $c$, and 
$\varepsilon$).  

On the other hand, since $M=DAD-\pi_0$, we see that
\begin{equation*}
 M-M' = \mu_e DKD - \pi_0. 
\end{equation*}
With respect to the orthonormal basis
$\{\mathbf{e}_i\}$, $\pi_0$ can be expressed as
\begin{equation*}
 \pi_0 = 
 \begin{pmatrix}
  \sqrt{\nu(v_1)} \\
  \vdots \\
  \sqrt{\nu(v_{(2m)^{n}})}
 \end{pmatrix}
 \begin{pmatrix}
  \sqrt{\nu(v_1)} & \cdots & \sqrt{\nu(v_{(2m)^{n}})}
 \end{pmatrix}
  = \left(
   \frac{\sqrt{\mathrm{deg}(v_i)\mathrm{deg}(v_j)}}
     {2\# R}
   \right), 
\end{equation*}
and hence
\begin{equation*}
 \begin{split}
 M-M' 
 & = \left(
  \frac{\mu_e}
     {\sqrt{\mathrm{deg}(v_i)\mathrm{deg}(v_j)}}
  - \frac{\sqrt{\mathrm{deg}(v_i)\mathrm{deg}(v_j)}}
     {2\# R}
 \right) \\
 & = \frac{1}{2\# R}
   \left(
    \frac{2\# R \mu_e -
          \mathrm{deg}(v_i)\mathrm{deg}(v_j)}
          {\sqrt{\mathrm{deg}(v_i)\mathrm{deg}(v_j)}}
   \right). 
 \end{split}
\end{equation*}
 Therefore we obtain
 \begin{equation*}
 \begin{split}
 & \Vert M-M' \Vert \\
 = & \max_{\vert f \vert=1} 
     \left|\left(f, (M-M')f\right) \right| \\
 = & \frac{1}{2\# R}
     \max_{\vert f \vert=1} \left| \sum_i f_i
    \left( \sum_j
    \frac{2\# R \mu_e - \mathrm{deg}(v_i)\mathrm{deg}(v_j)}
          {\sqrt{\mathrm{deg}(v_i)\mathrm{deg}(v_j)}} f_j
   \right) \right| \\
 = & \frac{1}{2\# R}
     \max_{\vert f \vert=1} \left| 
    \sum_i \sqrt{\mathrm{deg}(v_i)}f_i
    \left( \sum_j \left(
     \frac{2\# R \mu_e}{\mathrm{deg}(v_i)\mathrm{deg}(v_j)}  
     -1 \right)
      \sqrt{\mathrm{deg}(v_j)}f_j
   \right) \right|. 
 \end{split}
 \end{equation*}
 Recall that, by definition, $\# R$ is a binomially distributed random
 variable with distribution 
 $\mathrm{Bi}((2m)^{\eta}, (c_0(2m)^{d\eta})/(2m)^{\eta})$. Thus the
 expected value of $\# R$ is equal to $c_0(2m)^{d\eta}$. Noting that
\begin{equation*}
 2c_0(2m)^{d\eta} \mu_e
 = 2c_0(2m)^{d\eta} \times 2c_0 (2m)^{d\eta -2n}
 = 4c_0^2 (2m)^{2d\eta}(2m)^{-2n}
 = \mu_v^2 
\end{equation*}
and (\ref{eq:chernoff_bounds}), we see that, for any $\varepsilon > 0$, 
\begin{equation*}
 \begin{split}
  \mathbb{P}(\vert 2\# R \mu_e - \mu_v^2 \vert > \varepsilon \mu_v^2) 
 = &  \mathbb{P}(\vert \#R - c_0(2m)^{d\eta} \vert > 
              \varepsilon c_0(2m)^{d\eta}) \\
 \leq &  2 \exp \left(
     -\frac{\varepsilon^2 c_0(2m)^{d\eta}}{3}
    \right) \\
 \leq & 2 \exp \left(
    -\frac{\varepsilon^2c_0}{3} (2m)^{kl+\varepsilon_0 l}
    \right) \\
 \leq & b \exp \left(-b' (2m)^{\varepsilon_0 l}\right)
 \end{split} 
\end{equation*}
holds, where
\begin{equation*}
 b= 2, \quad b'=\frac{\varepsilon^2 c^{-1}}{3}. 
\end{equation*}
Note that we may assume $2a>b$ and $a'<b'$, which implies
\begin{equation}
\label{eq:sharp_R}
 \mathbb{P}(\vert 2\# R \mu_e - \mu_v^2 \vert > \varepsilon \mu_v^2) 
 < 2a \exp \left(-a' (2m)^{\varepsilon_0 l}\right). 
\end{equation}
By (\ref{eq:deg_conc}) and (\ref{eq:sharp_R}), we see that, for any
given $\varepsilon>0$, 
\begin{equation*}
 \begin{split}
 & \Vert M-M' \Vert  \\
 \leq & \frac{1}{2\# R}
     \max_{\vert f \vert=1} \left|
    \sum_i \sqrt{\mathrm{deg}(v_i)}f_i
    \left(\left(\frac{1+\varepsilon}{(1-\varepsilon)^2}-1\right) 
      \sum_j \sqrt{\mathrm{deg}(v_j)}f_j
   \right) \right| \\
 \leq & \frac{1}{2\# R}
     \max_{\vert f \vert=1} \sqrt{\sum_i \mathrm{deg}(v_i)}
     \vert f \vert 
    \left(\frac{3\varepsilon-\varepsilon^2}{(1-\varepsilon)^2}
     \sqrt{\sum_j \mathrm{deg}(v_j)} \vert f \vert
   \right)  \\
 = & \frac{1}{2\# R} \left(\frac{\varepsilon}{1-\varepsilon} + 
          \frac{2\varepsilon}{(1-\varepsilon)^2}\right)
    \sum_{i} \mathrm{deg}(v_i) \\
 = & \frac{\varepsilon}{1-\varepsilon} + 
          \frac{2\varepsilon}{(1-\varepsilon)^2}
 \end{split}
\end{equation*}
with probability $1-4a \exp (-a' (2m)^{\varepsilon_0 l})$,
where Cauchy-Schwarz inequality is used in the second inequality and 
$\sum_j \mathrm{deg}(v_j)=2\# E_P= 2\# R$ implies the last equality. 
Thus, together with (\ref{eq:M'_and_B}), for any $\varepsilon>0$, we
see that 
\begin{equation*}
 \Vert M \Vert \leq \Vert M' \Vert + \Vert M- M' \Vert 
 \leq \frac{\Vert B \Vert}{1- \varepsilon} +
 \frac{\varepsilon}{1-\varepsilon} + 
          \frac{2\varepsilon}{(1-\varepsilon)^2}
\end{equation*}
holds with probability greater than 
$1-6a \exp\left(-a' (2m)^{\varepsilon_0 l}\right)$, 
which is independent of $c_0$ (but depends on $m$, $k$,  $c$, and
$\varepsilon$).  Recall that (\ref{eq:prob_|B|}) tells us that we may
assume $\Vert B \Vert < \varepsilon$ with high probability. 
By taking $\varepsilon < 1/2$, we may assume
\begin{equation*}
   \frac{2\varepsilon}{1-\varepsilon} + 
          \frac{2\varepsilon}{(1-\varepsilon)^2} < 12\varepsilon. 
\end{equation*}
Thus we obtain, by taking (\ref{eq:prob_|B|}) in account,  
\begin{equation*}
 \mathbb{P}\left(
 \Vert M \Vert \leq 12\varepsilon
 \right)
 \geq
 1 - a_1 \exp\left(-a_2 (2m)^{\varepsilon_0 l}\right)
  - a_3(2m)^{-(\alpha \varepsilon_0 -k) l}, 
\end{equation*}
where
\begin{equation*}
 a_1=6a,\  a_2=a', \ a_3=\frac{c'(\alpha)}{\varepsilon^{2\alpha}}. 
\end{equation*}
By taking $\varepsilon= \varepsilon'/12$, this completes the proof of 
(\ref{eq:approx_eigenvalue}).

\section{Random quotients}
\label{sec:random-quotient}

In the proof of Propositions \ref{prop:random_eigenvalue} and
\ref{prop:random_Rayleigh_quotient}, we did not use the fact that
$\Gamma$ is a free group generated by $S$; $\Gamma$ can be any group
generated by $S$ (see \cite[IV.c]{Ollivier0} for related
comments). Therefore what we have actually proven is the 
following fixed-point theorem for random quotients:

\begin{Theorem}
 \label{thm:main_Hilbert_2}
 Let $\mathcal{H}$ be a Hilbert space.  Suppose that $\Gamma$ is a group
 generated by $S=\{s_1, \dots, s_m, s_1^{-1}, \dots, s_m^{-1}\}$, and
 let $\Gamma_P'=\Gamma/\bar{R}$ for $P \in P(m,\eta,d)$, where $\bar{R}$
 denotes the normal closure of the subset of $\Gamma$ determined by $R$. 
 For $1/3<d$, 
\begin{equation*}
 \lim_{\eta\to \infty} 
 \frac{\#\{P \in P(m,\eta,d) \mid 
        \Gamma_P' \text{ has }F(\mathcal{H})\}}
      {\# P(m,\eta,d)}= 1
\end{equation*}
 holds. 
\end{Theorem}

\begin{Theorem}
\label{thm:main_2}
 Let $Y$ be a $\cat$ space.  Suppose that $\Gamma$ is a group
 generated by $S=\{s_1, \dots, s_m, s_1^{-1}, \dots, s_m^{-1}\}$, and
 let $\Gamma_P'=\Gamma/\bar{R}$ for $P \in P(m,\eta,d)$, where $\bar{R}$
 denotes the normal closure of the 
 subset of $\Gamma$ determined by $R$. 
 Take $0\leq \delta<1$ and an integer $k$ so that $k>1/(1-\delta)$.  
 Let $\mathcal{Y}_{\delta}$ be the class of $\cat$ spaces with 
 $\delta(Y) \leq \delta$. 
 Then, for $k/(2k+1)<d$, 
\begin{equation*}
 \lim_{\eta\to \infty} 
 \frac{\#\{P \in P(m,\eta,d) \mid 
        \Gamma_P' \text{ has }F(Y) 
        \text{ for all }Y  \in \mathcal{Y}_{\delta}\}}
      {\# P(m,\eta,d)}= 1
\end{equation*}
 holds. 
\end{Theorem}

In the theorems above, we cannot say $\Gamma$
is infinite in general.  What is known in this direction is the
following theorem for hyperbolic groups due to Ollivier:

\begin{Theorem}[{\cite[Theorem 4]{Ollivier}}]
 Let $\Gamma$ be a hyperbolic group with harmless torsion generated by
 $S=\{s_1, \dots, s_m, s_1^{-1}, \dots, s_m^{-1}\}$ and $\theta$ the
 gross cogrowth of $\Gamma$.  Suppose that $d< 1-\theta$. Then we have
\begin{equation*}
 \lim_{\eta\to \infty} 
 \frac{\#\{P \in P(m,\eta,d) \mid \Gamma_P'
           \text{ is non-elementary hyperbolic}\}}
      {\# P(m,\eta,d)} = 1. 
\end{equation*}
\end{Theorem}

See \cite{Ollivier} for the definitions of {\it gross cogrowth} and 
{\it harmless torsion}. 
 We remark here that the gross cogrowth of the free group of rank $m$ is
 equal to $\frac{1}{2} \log_{2m}(8m-4)$.

\end{document}